\documentclass[10pt]{amsart}
\usepackage{cases}
\usepackage{cases}
\usepackage{amsmath}
\usepackage{color}
 \usepackage{graphicx}
 \usepackage{mathrsfs}
\usepackage{bbm}
\usepackage{amsfonts}
\usepackage{amssymb}
\usepackage{ams,amsmath,amsthm}
\usepackage{hyperref}
\usepackage{multirow}
\usepackage{slashbox}
\usepackage{makecell}
\def\opn#1#2{\def#1{\operatorname{#2}}} 
\opn\chara{char} \opn\length{\ell}
\opn\projdim{proj\,dim} \opn\injdim{inj\,dim} \opn\rank{rank}
\opn\depth{depth} \opn\grade{grade} \opn\height{height}
\opn\embdim{emb\,dim} \opn\codim{codim}

\opn\Tr{Tr} \opn\bigrank{big\,rank}
\opn\superheight{superheight}\opn\lcm{lcm}
\opn\trdeg{tr\,deg}%
\opn\reg{reg} \opn\lreg{lreg}
\opn\Ker{Ker} \opn\Coker{Coker} \opn\Im{Im} \opn\Hom{Hom}
\opn\Tor{Tor} \opn\Ext{Ext} \opn\End{End} \opn\Aut{Aut} \opn\id{id}

\opn\nat{nat}
\opn\pff{pf}
\opn\Pf{Pf} \opn\GL{GL} \opn\SL{SL} \opn\mod{mod} \opn\ord{ord}

\def\Implies{\ifmmode\Longrightarrow \else
     \unskip${}\Longrightarrow{}$\ignorespaces\fi}
\def\implies{\ifmmode\Rightarrow \else
     \unskip${}\Rightarrow{}$\ignorespaces\fi}
\def\iff{\ifmmode\Longleftrightarrow \else
     \unskip${}\Longleftrightarrow{}$\ignorespaces\fi}

\let\:=\colon

\newtheorem{Theorem}{Theorem}[section]
\newtheorem{Lemma}[Theorem]{Lemma}
\newtheorem{Assumption}[Theorem]{Assumption}

\newtheorem{Remark}[Theorem]{Remark}
\newtheorem{Algorithm}[Theorem]{Algorithm}

\newtheorem{Example}[Theorem]{Example}

\theoremstyle{definition}

\opn\ini{in} \opn\inm{inm} \opn\Sym{Sym} \opn\diag{diag}
\opn\Ii{(i)} \opn\Iii{(ii)}

\title{ A multilevel correction method for optimal controls of elliptic equation} 
\author{Wei Gong$^\diamond$}
\thanks{$^\diamond$ LSEC, Institute of Computational Mathematics, Academy of Mathematics and Systems
Science, Chinese Academy of Sciences, Beijing 100190, China
({\tt wgong@lsec.cc.ac.cn}).}
\author{Hehu Xie$^\dag$}
\thanks{$^\dag$NCMIS, LSEC, Institute of Computational Mathematics, Academy of Mathematics and Systems
Science, Chinese Academy of Sciences, Beijing 100190, China
({\tt hhxie@lsec.cc.ac.cn}).}
\author {Ningning Yan$^\ddag$}
\thanks{$^\ddag$NCMIS, LSEC, Institute of Systems Science, Academy of Mathematics and Systems Science,
 Chinese Academy of Sciences, Beijing 100190, China
({\tt ynn@amss.ac.cn}).}
\date{\today}
\begin{document}
\maketitle
{\bf Abstract:}\hspace*{10pt} {We propose in this paper a multilevel correction method
to solve optimal control problems constrained by elliptic equations with the finite
element method. In this scheme, solving optimization problem on the finest
finite element space is transformed to a series of solutions of linear boundary value problems
by the multigrid method on multilevel meshes and a series of solutions of
optimization problems on the coarsest finite element space. Our proposed scheme, instead of solving a 
large scale  optimization problem in the finest finite element space,  solves only
 a series of linear boundary value problems and the optimization
problems in a very low dimensional finite element space, and thus can improve the overall efficiency 
for the solution of optimal control problems governed by PDEs.}

{{\bf Keywords:}\hspace*{10pt} Optimal control problems, elliptic equation, control constraints,
finite element method, multilevel correction method  }


\section{Introduction}\setcounter{equation}{0}
Optimal control problems \cite{Hinze09book,Lions,LiuYan08book} play a very important role in
modern sciences and industries, and have many applications in such as chemical process, fluid dynamics, medicine,
 economics and so on. The finite element method is among the most important and popular numerical methods for solving
control problems governed by
partial differential equations. So far there have existed much work on the finite element
method for the optimal control problems. The interested readers are referred to
\cite{Falk,Geveci,LiuYan01SINUM,LiuYan01ACM,LiuYan08book}
and books and papers cited therein.

As we know, the control problems governed by partial differential equations
\cite{Hinze09book,Lions,LiuYan08book} are generally
nonlinear and result in large scale optimization problems which bring much more difficulties to design efficient solvers. It is
also well known that the multigrid or multilevel method is the optimal solver for many partial differential
equations discretized by the finite element method, finite difference method and so on (see, e.g.,  \cite{Shaidurovbook}).
Naturally, it is an important issue how to construct the multilevel type
numerical method for the optimal control problems governed by partial differential equations.
So far, there is only few work in this direction, we refer to \cite{Borzi09SIRE} for an overview.
 Since the classical multilevel or multigrid method for
the optimal control problem is designed to solve the linear algebraic systems formulated on each 
step of the optimization algorithm, it is not so easy to give the analysis on optimal error
estimates with the optimal computational complexity \cite{Borzi09SIRE}.

The aim of this paper is to propose a multilevel correction method for the
optimal control problems governed by partial differential equations
 based on the multilevel correction idea introduced in \cite{LinXie,Xie_IMA,Xie_JCP}.
In this method, solving the control problem will not be more difficult than solving
the corresponding linear boundary value problems. The multilevel correction method for the control
problem is based on a series of nested finite element spaces with different level
of accuracy which can be built with the same way as the multilevel method for boundary
value problems.
The multilevel correction scheme can be described as follows:
(1) solve the control problem in an initial coarse finite element space;
(2) use the multigrid method to solve two additional linear boundary value problems
which are constructed by using the previous obtained state, adjoint state and control approximations;
(3) solve a control problem again on the finite element space which is constructed by combining
the coarsest finite element space with the obtained state and adjoint state approximations in
step (2). Then go to step (2) for the next loop until stop. In this method, we replace
solving control problem on the finest finite element space by
solving a series of linear boundary value problems with multigrid scheme in the
corresponding series of finite element spaces and a series of control problems in the coarsest
finite element space.
The corresponding error and computational work estimates
of the proposed multilevel correction scheme for the control problem will also be analyzed.
Based on the analysis, the proposed method can obtain optimal errors with an almost optimal
computational complexity. So our proposed multilevel correction method can improve the overall efficiency
for solving the control problem as it does for linear boundary value problems.

An outline of the paper goes as follows. In Section 2, we introduce the finite element
method for the optimal control problem.
The multilevel correction method for the control problem is given in Sections 3.
In Section 4, we extend the multilevel correction method to the optimal control
problems governed by semilinear elliptic equation.  Section 5 is devoted to providing the numerical results to validate the
efficiency of the proposed numerical scheme.
Some concluding remarks are given in the last section.

\section{Finite element method for optimal control problem}\setcounter{equation}{0}
Let $\Omega\subset \mathbb{R}^{d}$, $d=2,3$ be a bounded and convex polygonal or polyhedral domain. 
Let $\|\cdot\|_{m,s,\Omega}$  and $\|\cdot\|_{m,\Omega}$ be the usual norms of the Sobolev spaces 
$W^{m,s}(\Omega)$ and $H^m(\Omega)$ respectively. Let  $|\cdot|_{m,s,\Omega}$ and $|\cdot|_{m,\Omega}$ be 
the usual seminorms of the above-mentioned two spaces respectively.

In this section, we introduce the finite element method for the optimal control problem constrained by
elliptic equations. The corresponding a  priori error estimates will also be given.

At first we consider the following linear-quadratic optimal control problem:
\begin{eqnarray}\label{OCP}
\min\limits_{u\in U_{ad}}\ \ J(y,u)={1\over
2}\|y-y_d\|_{0,\Omega}^2 +
\frac{\alpha}{2}\|u\|_{0,\Omega}^2
\end{eqnarray}
subject to
\begin{equation}\label{OCP_state}
\left\{\begin{array}{llr}
-\Delta y=f+u \ \ &\mbox{in}\ \Omega, \\
\ \ \ \ \ y=0  \ \ \ &\mbox{on}\ \partial\Omega.
\end{array} \right.
\end{equation}
The admissible control set is of box type:
\begin{eqnarray}
U_{ad}:=\Big\{u\in L^2(\Omega):\  a(x)\leq u(x)\leq b(x)\ \
 &\mbox{a.e.\ in}\ \Omega\Big\}\label{control_set}
\end{eqnarray}
with $a(x)< b(x)$ for a.e. $x\in\Omega$. We require $a,b\in L^\infty(\Omega)$.

Since the state equation (\ref{OCP_state}) is affine linear
with respect to the control $u$, we can introduce
 a linear operator $S:L^2(\Omega)\rightarrow H_0^1(\Omega)$ such that $y=Su+y_f$, where
 $y_f\in H_0^1(\Omega)$ is the solution of (\ref{OCP_state}) corresponding to the right hand side $f$.
  Then standard elliptic regularity theory gives $y\in H^2(\Omega)$. With this notation 
  we can formulate a reduced optimization problem
\begin{eqnarray}\label{OCP_Operator}
\min\limits_{u\in U_{ad}} \hat J(u):=J(Su,u)={1\over
2}\|Su+y_f-y_d\|_{0,\Omega}^2 +
\frac{\alpha}{2}\|u\|_{0,\Omega}^2.
\end{eqnarray}
Since the above optimization problem is linear and strictly convex,
there exists a unique solution
 $u\in U_{ad}$ (see \cite{Lions}). Moreover, the first order necessary and sufficient
 optimality condition can be stated as follows:
\begin{eqnarray}\label{reduced_opt}
J'(u)(v-u)=(\alpha u+S^*(Su+y_f-y_d),v-u)\geq 0, \
\ \ \ \ \forall v\in U_{ad},
\end{eqnarray}
where $S^*$ is the adjoint of $S$ (\cite{Hinze09book}).
Introducing the adjoint state $p=S^*(Su+y_f-y_d)\in H_0^1(\Omega)$,
we are led to the following optimality condition
\begin{equation}\label{OCP_OPT}
\left\{\begin{array}{llr}a(y,v)=(f+u,v),\ \ &\forall v\in H_0^1(\Omega),\\
a(w,p)=(y-y_d,w),\ \ &\forall w\in H_0^1(\Omega),\\
(\alpha u+p,v-u)\geq 0, \ \ &\forall v\in U_{ad},
\end{array} \right.
\end{equation}
where we use the standard notations
\begin{eqnarray}
a(y,v):=\int_\Omega\nabla y\nabla vdx,\ \ &\forall y,v\in H_0^1(\Omega).\nonumber\\
(y,v):=\int_\Omega yvdx,\ \ &\forall y,v\in L^2(\Omega).\nonumber
\end{eqnarray}
Hereafter, we call $u$, $y$ and $p$ the optimal control, state and adjoint state, respectively.

With the admissible control set (\ref{control_set}) we can get the
explicit representation of the optimal control $u$
through the adjoint state $p$
\begin{eqnarray}
u(x)=P_{U_{ad}}\big\{-\frac{1}{\alpha}p(x)\big\},\label{p_to_u}
\end{eqnarray}
where $P_{U_{ad}}$ is the orthogonal projection operator onto $U_{ad}$.

Let $\mathcal{T}_h$ be a regular and quasi-uniform triangulation of $\Omega$ such that
$\bar\Omega=\cup_{\tau\in\mathcal{T}_h}\bar\tau$. On $\mathcal{T}_h$
we construct the piecewise linear and continuous finite element space $V_{h}$
such that $V_h\subset C(\bar\Omega)\cap H_0^1(\Omega)$.
Based on the finite element space $V_h$,
we can define the finite dimensional approximation to the optimal
control problem (\ref{OCP})-(\ref{OCP_state}) as follows: Find
$(\bar{u}_h,\bar{y}_h,\bar{p}_h)\in U_{ad}\times V_h\times V_h$ such that
\begin{eqnarray}\label{OCP_h}
\min\limits_{\bar u_h\in U_{ad}}J_h(\bar y_h,\bar u_h)={1\over
2}\|\bar y_h-y_d\|_{0,\Omega}^2 +
\frac{\alpha}{2}\|\bar u_h\|_{0,\Omega}^2
\end{eqnarray}
subject to
\begin{eqnarray}\label{OCP_state_h}
a(\bar y_h,v_h)=(f+\bar u_h,v_h),\ \ \ \ \forall v_h\in V_h.
\end{eqnarray}
In this paper, we use the piecewise linear finite element
 to approximate the state $y$, and variational discretization for the optimal
 control $u$ (see \cite{Hinze05COAP}).
Similar to the infinite dimensional problem (\ref{OCP})-(\ref{OCP_state}),
the above discretized optimization problem also admits a unique
solution $\bar{u}_h\in U_{ad}$. The discretized first order necessary and sufficient optimality
condition can be stated as follows:
\begin{equation}\label{OCP_OPT_h}
\left\{\begin{array}{llr}a(\bar y_h,v_h)=(f+\bar u_h,v_h),\ \ &\forall v_h\in V_h,\\
 a(w_h,\bar p_h)=(\bar y_h-y_d,w_h),\ \ &\forall w_h\in V_h,\\
 (\alpha \bar u_h+\bar p_h,v_h-\bar u_h)\geq 0, \ \ \ &\forall v_h\in U_{ad}.
\end{array} \right.
\end{equation}
The above optimization problem can be solved by projected gradient method
or semi-smooth Newton method, see \cite{Hintermueller03SIOPT}, \cite{Hinze09book}, \cite{Vierling} and \cite{LiuYan08book}
 for more details.

Now we state the following error estimate results for the finite element approximation
of the control problem and the proof can be found in \cite{Hinze05COAP}.
\begin{Theorem}\label{Thm:2.1}
Let $(u,y,p)\in U_{ad}\times H_0^1(\Omega)\times H_0^1(\Omega)$ and
 $(\bar u_h,\bar y_h,\bar p_h)\in U_{ad}\times V_h\times V_h$ be the
 solutions of problems (\ref{OCP})-(\ref{OCP_state}) and
 (\ref{OCP_h})-(\ref{OCP_state_h}), respectively.
 Then the following error estimates hold
\begin{eqnarray}\label{u_error}
\|u-\bar u_h\|_{0,\Omega}+\|y-\bar y_h\|_{0,\Omega}
+\|p-\bar p_h\|_{0,\Omega}\leq Ch^{2}.
\end{eqnarray}
\end{Theorem}

\section{Multilevel correction method for optimal control problems}
\setcounter{equation}{0}
In this section, we propose a type of multilevel correction method for the optimal
control problem (\ref{OCP_h})-(\ref{OCP_state_h}).
In this scheme, solving the optimization problem on the finest finite element spaces
is transformed to a series of solutions of linear boundary value problems by the
 multigrid method on multilevel meshes and a series of solutions of
  optimization problems on the coarsest finite element space.

In order to introduce the multilevel correction scheme, we define a sequence of
 triangulations $\mathcal{T}_{h_k}$ of $\Omega$ determined as follows.
Suppose a very coarse mesh $\mathcal{T}_{H}$ is given
and let $\mathcal{T}_{h_k}$ be obtained from $\mathcal{T}_{h_{k-1}}$  via regular
refinement (produce $\beta^d$ subelements) such that
$$h_k\approx\frac{1}{\beta}h_{k-1}$$
for $k=1,\cdots, n$ and $\mathcal{T}_{h_0}:=\mathcal{T}_H$. Here $\beta\geq 2$ is a positive integer.

Let $V_H$ denote the coarsest linear finite element space defined on the coarsest
mesh $\mathcal{T}_H$. Besides, we construct a series of finite element spaces $V_{h_1}$,
$V_{h_2}$, $\cdots$, $V_{h_n}$ defined on the corresponding series of multilevel meshes
$\mathcal{T}_{h_k}$ ($k=1,2,\cdots,n$) such that
 $V_H\subseteq V_{h_1}\subset V_{h_2}\subset\cdots\subset V_{h_n}$.

In order to design the multilevel correction method for the optimization problem, we
first introduce an one correction step which can improve the accuracy of the
given numerical approximations for the state, adjoint state and optimal control. This correction step contains
solving two linear boundary value problems with multigrid method in the finer finite
element space and an optimization problem on the coarsest finite element space.

Assume that we have obtained an approximate solution
$(u_{h_{k}},y_{h_{k}},p_{h_{k}})\in U_{ad}\times V_{h_{k}}\times V_{h_{k}}$ on the
$k$-th level mesh $\mathcal{T}_{h_k}$.
Now we introduce an one correction step to improve the accuracy of
 the current approximation $(u_{h_{k}},y_{h_{k}},p_{h_{k}})$.
\begin{Algorithm}\label{Alg:3.1} One correction step:
\begin{enumerate}
\item Find $y_{h_{k+1}}^*\in V_{h_{k+1}}$ such that
\begin{equation}\label{step_2}
a(y^*_{h_{k+1}},v_{h_{k+1}}) = (f+u_{h_{k}},v_{h_{k+1}}),\ \
\ \ \forall\ v_{h_{k+1}}\in V_{h_{k+1}}.
\end{equation}
Solve the above equation with multigrid method to obtain an approximation
$\hat y_{h_{k+1}}\in V_{h_{k+1}}$ with error
$\|\hat y_{h_{k+1}}-y_{h_{k+1}}^*\|_{1,\Omega}\leq Ch_{h_k}^{2}$
and define $\hat{y}_{h_{k+1}}:=MG(V_{h_{k+1}},u_{h_k})$.

\item Find $p_{h_{k+1}}^*\in V_{h_{k+1}}$ such that
\begin{equation}\label{step_3}
a(w_{h_{k+1}},p^*_{h_{k+1}}) = (\hat y_{h_{k+1}}-y_d,v_{h_{k+1}}),\ \
\ \ \forall\ v_{h_{k+1}}\in V_{h_{k+1}}.
\end{equation}
Solve the above equation with multigrid method to obtain an approximation
 $\hat p_{h_{k+1}}\in V_{h_{k+1}}$ with error
 $\|\hat p_{h_{k+1}}-p_{h_{k+1}}^*\|_{1,\Omega}\leq Ch_{h_k}^{2}$
 and define $\hat{p}_{h_{k+1}}:=MG(V_{h_{k+1}},\hat y_{h_{k+1}})$.

\item Define a new finite element space
$V_{H,h_{k+1}}:=V_H+{\rm span}\{\hat y_{h_{k+1}}\}+{\rm span}\{\hat p_{h_{k+1}}\}$
and solve the following optimal control problem:
\begin{eqnarray}\label{step_4}
\min\limits_{u_{h_{k+1}}\in U_{ad},\ y_{h_{k+1}}\in V_{H,h_{k+1}}}
 J(y_{h_{k+1}},u_{h_{k+1}})={1\over
2}\|y_{h_{k+1}}-y_d\|_{0,\Omega}^2 +
\frac{\alpha}{2}\|u_{h_{k+1}}\|_{0,\Omega}^2
\end{eqnarray}
subject to
\begin{equation}\label{step_4_state}
a(y_{h_{k+1}},v_{H,h_{k+1}}) = (f+u_{h_{k+1}},v_{H,h_{k+1}}),
\ \ \ \ \forall\ v_{H,h_{k+1}}\in V_{H,h_{k+1}}.
\end{equation}
The corresponding optimality condition reads:
 Find $(u_{h_{k+1}},y_{h_{k+1}},p_{h_{k+1}})
 \in U_{ad}\times V_{H,h_{k+1}}\times V_{H,h_{k+1}}$
 such that
\begin{equation}\label{step_4_OPT}
\left\{\begin{array}{llr}
a(y_{h_{k+1}},v_{H,h_{k+1}}) = (f+u_{h_{k+1}},v_{H,h_{k+1}}),
\ \ &\forall\ v_{H,h_{k+1}}\in V_{H,h_{k+1}},\\
a(v_{H,h_{k+1}},p_{h_{k+1}}) = (y_{h_{k+1}}-y_d,v_{H,h_{k+1}}),
\ \ &\forall\ v_{H,h_{k+1}}\in V_{H,h_{k+1}},\\
(\alpha u_{h_{k+1}}+p_{h_{k+1}},v-u_{h_{k+1}})\geq 0,
\ \ &\forall\ v\in U_{ad}.
\end{array} \right.
\end{equation}
\end{enumerate}
We define the output of above algorithm as
\begin{eqnarray}\label{correction}
(u_{h_{k+1}},y_{h_{k+1}},p_{h_{k+1}})
=\mbox{\rm Correction}(V_H,u_{h_{k}},y_{h_k},p_{h_k},V_{h_{k+1}}),
\end{eqnarray}
where $V_H$ denotes the coarsest finite element space, $(u_{h_{k}}, y_{h_k}, p_{h_k})$
is the given approximation of the optimal control, the state and the adjoint state
in the coarse finite element space $V_{h_k}$ and $V_{h_{k+1}}$ denotes the finer finite
element space.
\end{Algorithm}
\begin{Remark}
In Algorithm \ref{Alg:3.1}, one needs to solve an optimization problem (\ref{step_4})-(\ref{step_4_state})
 on the finite element space
$V_{H,h_{k+1}}$, there are several ways to provide a good initial guess for the optimization algorithm 
which may speed up the convergence. One option is to use $u_{h_{k}}$ as initial guess, while the other
 choice is $P_{U_{ad}}\{-\frac{1}{\alpha}\hat p_{h_{k+1}}\}$.
\end{Remark}
In the following of this paper, we denote
$(\bar u_{h_{k+1}},\bar y_{h_{k+1}},\bar p_{h_{k+1}})\in U_{ad}
\times V_{h_{k+1}}\times V_{h_{k+1}}$ the finite element solution to the
discrete optimal control problems (\ref{OCP_h})-(\ref{OCP_state_h})
in the finite element space $V_{h_{k+1}}$.
We are able to analyze the error estimates between solutions
 $(\bar u_{h_{k+1}},\bar y_{h_{k+1}},\bar p_{h_{k+1}})$
 and the correction one $(u_{h_{k+1}},y_{h_{k+1}},p_{h_{k+1}})$
 on mesh level $\mathcal{T}_{h_{k+1}}$.
\begin{Theorem}\label{Thm:3.1}
Let $(\bar u_{h_{k+1}},\bar y_{h_{k+1}},
\bar p_{h_{k+1}})\in U_{ad}\times V_{h_{k+1}}\times V_{h_{k+1}}$ be
the solution of problems (\ref{OCP_h})-(\ref{OCP_state_h}) and
$(u_{h_{k+1}},y_{h_{k+1}},p_{h_{k+1}})$ be the numerical approximation
by Algorithm \ref{Alg:3.1}, respectively.
Assume there exists a real number $\eta_{h_k}$ such that
$(u_{h_k}, y_{h_k}, p_{h_k})$ have the following error estimates
\begin{eqnarray}\label{assume_error}
\|\bar u_{h_k}-u_{h_k}\|_{0,\Omega}+\|\bar y_{h_k}-y_{h_k}\|_{0,\Omega}
+\|\bar p_{h_k}-p_{h_k}\|_{0,\Omega} = \eta_{h_k}.
\end{eqnarray}
Then the following error estimates hold
\begin{eqnarray}\label{correct_error}
\|\bar u_{h_{k+1}}-u_{h_{k+1}}\|_{0,\Omega}
+\|\bar y_{h_{k+1}}-y_{h_{k+1}}\|_{0,\Omega}
+\|\bar p_{h_{k+1}}-p_{h_{k+1}}\|_{0,\Omega}
\leq C\eta_{h_{k+1}},
\end{eqnarray}
where $\eta_{h_{k+1}}=H(h_{k}^{2}+\eta_{h_{k}})$.
\end{Theorem}
\begin{proof}
Note that
\begin{eqnarray*}
a(\bar y_{h_{k+1}},v_{h_{k+1}})=(f+\bar u_{h_{k+1}},v_{h_{k+1}}),\ \
\ \ \forall v_{h_{k+1}}\in V_{h_{k+1}},
\end{eqnarray*}
we conclude from (\ref{step_2}) that
\begin{eqnarray*}
a(\bar y_{h_{k+1}}-y^*_{h_{k+1}},v_{h_{k+1}})
=(\bar u_{h_{k+1}}-u_{h_{k}},v_{h_{k+1}}),\ \
\ \ \forall v_{h_{k+1}}\in V_{h_{k+1}},
\end{eqnarray*}
which implies
\begin{eqnarray*}
\|\bar y_{h_{k+1}}-y^*_{h_{k+1}}\|_{1,\Omega}
&\leq& C\|\bar u_{h_{k+1}}-u_{h_{k}}\|_{0,\Omega}\nonumber\\
&\leq& C\|\bar u_{h_{k+1}}-\bar u_{h_{k}}\|_{0,\Omega}
+C\|\bar u_{h_{k}}-u_{h_{k}}\|_{0,\Omega}\nonumber\\
&\leq& C(h_k^{2}+\eta_{h_k}).
\end{eqnarray*}
Note that $\hat y_{h_{k+1}}$ is obtained by the multigrid method with estimate
$\|\hat y_{h_{k+1}}-y_{h_{k+1}}^*\|_{1,\Omega}\leq
 Ch_{h_k}^{2}$, triangle inequality yields
\begin{eqnarray}\label{correct_error_2}
\|\bar y_{h_{k+1}}-\hat y_{h_{k+1}}\|_{1,\Omega}&\leq&
\|\bar y_{h_{k+1}}-y^*_{h_{k+1}}\|_{1,\Omega}
+\|y^*_{h_{k+1}}-\hat y_{h_{k+1}}\|_{1,\Omega}\nonumber\\
&\leq& C(h_k^{2}+\eta_{h_k}).
\end{eqnarray}
Similarly, we can prove
\begin{eqnarray}\label{correct_error_3}
\|\bar p_{h_{k+1}}-\hat p_{h_{k+1}}\|_{1,\Omega}
&\leq&C(h_k^{2}+\eta_{h_k}).
\end{eqnarray}
From (\ref{OCP_OPT_h}) and (\ref{step_4_OPT}), we have
\begin{eqnarray*}
(\alpha \bar u_{h_{k+1}}+
\bar p_{h_{k+1}},v_{h_{k+1}}-\bar u_{h_{k+1}})
\geq 0,\ \ \ \ \forall v_{h_{k+1}}\in U_{ad}
\end{eqnarray*}
and
\begin{eqnarray*}
(\alpha u_{h_{k+1}}+ p_{h_{k+1}},w_{h_{k+1}}-u_{h_{k+1}})\geq 0,
\ \ \ \ \forall w_{h_{k+1}}\in U_{ad}.
\end{eqnarray*}
Setting $v_{h_{k+1}}=u_{h_{k+1}}$ and $w_{h_{k+1}}=\bar u_{h_{k+1}}$,
adding the above two inequalities together we are led to
\begin{eqnarray}\label{correct_error_4}
&&\alpha\|\bar u_{h_{k+1}}-u_{h_{k+1}}\|_{0,\Omega}^2\nonumber\\
&\leq& (\bar u_{h_{k+1}}-u_{h_{k+1}},p_{h_{k+1}}-\bar p_{h_{k+1}})\nonumber\\
&=&\big(\bar u_{h_{k+1}}-u_{h_{k+1}},p_{h_{k+1}}-p_{h_{k+1}}(\bar y_{h_{k+1}})\big)
+\big(\bar u_{h_{k+1}}-u_{h_{k+1}},p_{h_{k+1}}(\bar y_{h_{k+1}})-\bar p_{h_{k+1}}\big)\nonumber\\
&=&a\big(y_{h_{k+1}}(\bar u_{h_{k+1}})-y_{h_{k+1}},p_{h_{k+1}}-p_{h_{k+1}}(\bar y_{h_{k+1}})\big)\nonumber\\
&&\ \ \ \ +\big(\bar u_{h_{k+1}}-u_{h_{k+1}},p_{h_{k+1}}(\bar y_{h_{k+1}})-\bar p_{h_{k+1}}\big)\nonumber\\
&=&\big(y_{h_{k+1}}(\bar u_{h_{k+1}})-y_{h_{k+1}},y_{h_{k+1}}-\bar y_{h_{k+1}}\big)\nonumber\\
&&\ \ \ \ +\big(\bar u_{h_{k+1}}-u_{h_{k+1}},p_{h_{k+1}}(\bar y_{h_{k+1}})-\bar p_{h_{k+1}}\big),
\end{eqnarray}
where $y_{h_{k+1}}(\bar u_{h_{k+1}})\in V_{H,h_{k+1}}$ and
$p_{h_{k+1}}(\bar y_{h_{k+1}})\in V_{H,h_{k+1}}$ satisfy
\begin{eqnarray*}
a(y_{h_{k+1}}(\bar u_{h_{k+1}}),v_{H,h_{k+1}})=(f+\bar u_{h_{k+1}},v_{H,h_{k+1}}),
\ \ \ \ \forall v_{H,h_{k+1}}\in V_{H,h_{k+1}}
\end{eqnarray*}
and
\begin{eqnarray*}
a(w_{H,h_{k+1}},p_{h_{k+1}}(\bar y_{h_{k+1}}))=(\bar y_{h_{k+1}}-y_d,w_{H,h_{k+1}}),
\ \ \ \ \forall w_{H,h_{k+1}}\in V_{H,h_{k+1}}.
\end{eqnarray*}
Then triangle inequality and $\epsilon$-Young inequality yield
\begin{eqnarray}\label{correct_error_5}
&&\|\bar u_{h_{k+1}}-u_{h_{k+1}}\|_{0,\Omega}+
\|\bar y_{h_{k+1}}-y_{h_{k+1}}\|_{0,\Omega}\nonumber\\
&\leq& C\big(\|y_{h_{k+1}}(\bar u_{h_{k+1}})-\bar y_{h_{k+1}}\|_{0,\Omega}
+\|p_{h_{k+1}}(\bar y_{h_{k+1}})-\bar p_{h_{k+1}}\|_{0,\Omega}\big).
\end{eqnarray}
It is easy to see that $y_{h_{k+1}}(\bar u_{h_{k+1}})$ is the finite element
 approximation to $\bar y_{h_{k+1}}$ on $V_{H,h_{k+1}}$ because of $V_{H,h_{k+1}}\subset V_{h_{k+1}}$.
  Standard Ce\'{a}-lemma implies (cf. \cite{Brenner})
\begin{eqnarray}
\|y_{h_{k+1}}(\bar u_{h_{k+1}})-\bar y_{h_{k+1}}\|_{1,\Omega}
&\leq& C\inf\limits_{v_{H,h_{k+1}}\in V_{H,h_{k+1}}}\|\bar y_{h_{k+1}}-v_{H,h_{k+1}}\|_{1,\Omega}\nonumber\\
&\leq&C\|\bar y_{h_{k+1}}-\hat y_{h_{k+1}}\|_{1,\Omega}\nonumber\\
&\leq& C(h_k^{2}+\eta_{h_k}).\label{correct_error_6}
\end{eqnarray}
From the following equation
\begin{eqnarray*}
a(y_{h_{k+1}}(\bar u_{h_{k+1}})-\bar y_{h_{k+1}},v_{H,h_{k+1}})=0,
\ \ \ \ \forall v_{H,h_{k+1}}\in V_{H,h_{k+1}},
\end{eqnarray*}
and Aubin-Nitsche technique (cf. \cite{Brenner}), we are able to prove the improved $L^2$-norm estimate
\begin{eqnarray}\label{correct_error_7}
\|y_{h_{k+1}}(\bar u_{h_{k+1}})-\bar y_{h_{k+1}}\|_{0,\Omega}&\leq&
CH\|y_{h_{k+1}}(\bar u_{h_{k+1}})-\bar y_{h_{k+1}}\|_{1,\Omega}\nonumber\\
&\leq&CH(h_k^{2}+\eta_{h_k}).
\end{eqnarray}
Similarly to (\ref{correct_error_6})-(\ref{correct_error_7}), we can derive
\begin{eqnarray}\label{correct_error_8}
\|p_{h_{k+1}}(\bar y_{h_{k+1}})-\bar p_{h_{k+1}}\|_{0,\Omega}
&\leq&CH(h_k^{2}+\eta_{h_k}).
\end{eqnarray}
Combining (\ref{correct_error_5}), (\ref{correct_error_7}) and (\ref{correct_error_8}) leads to
the following estimate
\begin{eqnarray}\label{correct_error_9}
\|\bar u_{h_{k+1}}-u_{h_{k+1}}\|_{0,\Omega}+\|\bar y_{h_{k+1}}-y_{h_{k+1}}\|_{0,\Omega}
&\leq&CH(h_k^{2}+\eta_{h_k}).
\end{eqnarray}
Using the triangle inequality, we obtain
\begin{eqnarray*}
\|\bar u_{h_{k+1}}-u_{h_{k+1}}\|_{0,\Omega}+\|\bar y_{h_{k+1}}-y_{h_{k+1}}\|_{0,\Omega}
+\|\bar p_{h_{k+1}}-p_{h_{k+1}}\|_{0,\Omega}
\leq CH(h_k^{2}+\eta_{h_k}),
\end{eqnarray*}
which is the desired result (\ref{correct_error}) and the proof is complete.
\end{proof}

Based on the sequence of nested finite element spaces
$V_{h_1}\subset V_{h_2}\subset\cdots\subset V_{h_n}$
and the one correction step defined in Algorithm \ref{Alg:3.1},
we can define the following multilevel correction method to solve the optimal control problem:
\begin{Algorithm}\label{Alg:3.2}
A multilevel correction method for optimal control problem:
\begin{enumerate}
\item Solve an optimal control problem in the initial finite element space $V_{h_1}$:
\begin{eqnarray}\label{step_1}
\min\limits_{u_{h_1}\in U_{ad},y_{h_1}\in V_{h_1}} J(y_{h_1},u_{h_1})={1\over
2}\|y_{h_1}-y_d\|_{0,\Omega}^2 +
\frac{\alpha}{2}\|u_{h_1}\|_{0,\Omega}^2
\end{eqnarray}
subject to
\begin{equation}\label{step_1_state}
a(y_{h_1},v_{h_1}) = (f+u_{h_1},v_{h_1}),\ \ \ \ \forall\ v_{h_1}\in V_{h_1}.
\end{equation}
The corresponding optimality condition reads:
Find $(u_{h_1},y_{h_1},p_{h_1})\in U_{ad}\times V_{h_1}\times V_{h_1}$ such that
\begin{equation}\label{step_1_OPT}
\left\{\begin{array}{llr}
a(y_{h_1},v_{h_1}) = (f+u_{h_1},v_{h_1}),\ \ &\forall\ v_{h_1}\in V_{h_1},\\
a(v_{h_1},p_{h_1}) = (y_{h_1}-y_d,v_{h_1}),\ \ &\forall\ v_{h_1}\in V_{h_1},\\
(\alpha u_{h_1}+p_{h_1},v-u_{h_1})\geq 0,\ \ &\forall\ v\in U_{ad}.
\end{array}
\right.
\end{equation}

\item Do $k=1$, $\cdots$, $n-1$
	
Obtain a new optimal solution $(u_{h_{k+1}},y_{h_{k+1}},p_{h_{k+1}})
\in U_{ad}\times V_{h_{k+1}}\times V_{h_{k+1}}$ by Algorithm \ref{Alg:3.1}
\begin{eqnarray*}
(u_{h_{k+1}},y_{h_{k+1}},p_{h_{k+1}})={\rm Correction}(V_H,u_{h_{k}},y_{h_k},p_{h_k},V_{h_{k+1}}).
\end{eqnarray*}
end Do
\end{enumerate}
Finally, we obtain a numerical approximation $(u_{h_n},y_{h_n}, p_{h_n})
\in U_{ad}\times V_{h_{n}}\times V_{h_n}$ for problem (\ref{OCP})-(\ref{OCP_state}).
\end{Algorithm}
Now we are in the position to give the error estimates for the solution generated by the above
 multilevel correction scheme described in Algorithm \ref{Alg:3.2}.
\begin{Theorem}\label{Thm:3.2}
Let $(u,y,p)\in U_{ad}\times H_0^1(\Omega)\times H_0^1(\Omega)$ and
 $(u_{h_n},y_{h_n},p_{h_n})\in U_{ad}\times V_{h_n}\times V_{h_n}$
 be the solution of problems (\ref{OCP})-(\ref{OCP_state})
 and the solution by Algorithm \ref{Alg:3.2}, respectively.
Assume the mesh size $H$ satisfies the condition  $CH\beta^2<1$.
Then the following error estimates hold
\begin{eqnarray}\label{multigrid_error}
\|\bar{u}_{h_n}-u_{h_n}\|_{0,\Omega}+\|\bar{y}_{h_n}-y_{h_n}\|_{0,\Omega}
+\|\bar{p}_{h_n}-p_{h_n}\|_{0,\Omega}\leq Ch_n^{2}.
\end{eqnarray}
Finally, we have the following error estimates
\begin{eqnarray}\label{final_error}
\|u-u_{h_n}\|_{0,\Omega}+\|y-y_{h_n}\|_{0,\Omega}+\|p-p_{h_n}\|_{0,\Omega}\leq
Ch_n^{2}.
\end{eqnarray}
\end{Theorem}
\begin{proof}
Since we solve the optimal control problem directly in the first step of
Algorithm \ref{Alg:3.2}, we have the following estimates
\begin{eqnarray}\label{Estimate_h_1}
\|\bar{u}_{h_1}-u_{h_1}\|_{0,\Omega}+\|\bar{y}_{h_1}-y_{h_1}\|_{0,\Omega}+\|\bar{p}_{h_1}-p_{h_1}\|_{0,\Omega}=0.
\end{eqnarray}
From Theorem \ref{Thm:3.1} and its proof, the following estimates for $(u_{h_2}, y_{h_2}, p_{h_2})$
hold
\begin{eqnarray}\label{Estimate_h_2}
\|\bar{u}_{h_2}-u_{h_2}\|_{0,\Omega}+\|\bar{y}_{h_2}-y_{h_2}\|_{0,\Omega}
+\|\bar{p}_{h_2}-p_{h_2}\|_{0,\Omega}\leq CH h_1^{2}.
\end{eqnarray}
Then based on Theorem \ref{Thm:3.1}, the condition $CH\beta^2<1$
 and recursive argument, we have
\begin{eqnarray}
\eta_{h_n}&\leq& CH\big(h_{n-1}^{2}+\eta_{h_{n-1}}\big)\leq CH\big(h_{n-1}^{2}+CH(h_{n-2}^{2}+\eta_{h_{n-2}})\big)\nonumber\\
&\leq& \sum_{k=1}^{n-1}(CH)^{(n-k)}h_{k}^{2} \leq \Big(\sum_{k=1}^{n-1}(CH)^{(n-k)}\beta^{2(n-k)}\Big)h_{n}^{2} \nonumber\\
&=& \Big(\sum_{k=1}^{n-1}(CH\beta^2)^{(n-k)}\Big)h_{n}^{2}
\leq \frac{CH\beta^2}{1-(CH\beta^2)}h_n^{2}\leq Ch_n^{2}.
\end{eqnarray}
This is the desired result (\ref{multigrid_error}) and the estimate (\ref{final_error})
can be derived by combining (\ref{multigrid_error}) and (\ref{u_error}).
Then the proof is complete.
\end{proof}
Now, we come to analyze the computational work for the multilevel correction
scheme defined in Algorithm \ref{Alg:3.2}. Since the linear
boundary value problems (\ref{step_2}) and (\ref{step_3})
in Algorithm \ref{Alg:3.1} are solved by multigrid method,
the corresponding computational work is of optimal order.

We define the dimension of each level linear finite element space as
\begin{eqnarray*}
N_k := {\rm dim}\ V_{h_k},\ \ \ k=1,\cdots,n.
\end{eqnarray*}
Then the following relation holds
\begin{eqnarray}\label{relation_dimension}
N_k \thickapprox\Big(\frac{1}{\beta}\Big)^{d(n-k)}N_n,\ \ \ k=1,\cdots,n.
\end{eqnarray}

The estimate of computational work for the second step in Algorithm
\ref{Alg:3.1} is different from the linear eigenvalue problems
\cite{LinXie,Xie_IMA,Xie_JCP}.
In this step, we need to solve a constrained optimization problem (\ref{step_4_OPT}). Always, 
some types of optimization methods are used to solve this problem.
In each iteration step, we need to evaluate the orthogonal projection
in the finite element space $V_{H,h_k}$ ($k=2,\cdots,n$) onto $U_{ad}$ which needs work $\mathcal{O}(N_k)$.
Fortunately, this step always can be carried out in the parallel way. 
\begin{Theorem}\label{Thm:linear}
Assume that we solve Algorithm \ref{Alg:3.2} with $m$ processors parallely,
the optimization problem solving in the coarse spaces $V_{H,h_k}$ ($k=1,\cdots, n$)
and $V_{h_1}$ need work $\mathcal{O}(M_H)$ and $\mathcal{O}(M_{h_1})$, respectively, and
the work of multigrid method for solving the boundary value problems in $V_{h_k}$ is $\mathcal{O}(N_k)$
for $k=2,3,\cdots,n$. Let $\varpi$ denote the iteration number of the optimization algorithm when we solve 
the optimization problem (\ref{step_4_OPT}) in the coarse space.
Then in each computational processor, the work involved
in Algorithm \ref{Alg:3.2} has the following estimate
\begin{eqnarray}\label{Computation_Work_Estimate}
{\rm Total\ work}&=&\mathcal{O}\Big(\big(1+\frac{\varpi}{m}\big)N_n
+ M_H\log N_n+M_{h_1}\Big).
\end{eqnarray}
\end{Theorem}
\begin{proof}
In each computational processor, let $W_k$ denote the computational work for the correction step in  the $k$-th finite element space $V_{h_k}$.
Then from the description of Algorithm  \ref{Alg:3.1} we have
\begin{eqnarray}\label{work_k}
W_k&=&\mathcal{O}\left(N_k +M_H+\varpi\frac{N_k}{m}\right), \ \ \ \ {\rm for}\  k=2,\cdots,n.
\end{eqnarray}
Iterating (\ref{work_k}) and using (\ref{relation_dimension}), we obtain
\begin{eqnarray}\label{Work_Estimate}
\text{Total work} &=& \sum_{k=1}^nW_k\nonumber =
\mathcal{O}\left(M_{h_1}+\sum_{k=2}^n
\Big(N_k + M_H+\varpi\frac{N_k}{m}\Big)\right)\nonumber\\
&=& \mathcal{O}\Big(\sum_{k=2}^n\Big(1+\frac{\varpi}{m}\Big)N_k
+ (n-1) M_H + M_{h_1}\Big)\nonumber\\
&=& \mathcal{O}\left(\sum_{k=2}^n
\Big(\frac{1}{\beta}\Big)^{d(n-k)}\Big(1+\frac{\varpi}{m}\Big)N_n
+ M_H\log N_n+M_{h_1}\right)\nonumber\\
&=& \mathcal{O}\left(\big(1+\frac{\varpi}{m}\big)N_n
+ M_H\log N_n+M_{h_1}\right).
\end{eqnarray}
This is the desired result and we complete the proof.
\end{proof}
\begin{Remark}
Since we have a good enough initial solution $(\hat{y}_{h_{k+1}},\hat p_{h_{k+1}})$
in the second step of Algorithm \ref{Alg:3.1},
solving the optimization problem (\ref{step_4_OPT}) always does not
need too many iterations. Then the complexity in each computational
node is always $\mathcal{O}(N_n)$ provided $M_H\ll N_n$ and $M_{h_1}\leq N_n$.
\end{Remark}

\section{Application to optimal controls of semilinear elliptic equation}\setcounter{equation}{0}
In this section, we will extend the multilevel correction method
to optimal control problem governed by semilinear elliptic equation:
\begin{eqnarray}\label{OCP_nonlinear}
\min\limits_{u\in U_{ad}}\ \ J(y,u)={1\over
2}\|y-y_d\|_{0,\Omega}^2 +
\frac{\alpha}{2}\|u\|_{0,\Omega}^2
\end{eqnarray}
subject to
\begin{equation}\label{OCP_state_nonlinear}
\left\{\begin{array}{llr} -\Delta y+\phi(\cdot,y)=f+u \ \ &\mbox{in}\
\Omega, \\
\ \ \ \ \ \ \ \ \ \ \ \ \ \ y=0  \ \ \ &\mbox{on}\ \partial\Omega,
\end{array}
\right.
\end{equation}
where the function $\phi:\Omega\times\mathbb{R}\rightarrow \mathbb{R}$ is measurable with respect to $x\in \Omega$ for all $y\in \mathbb{R}$ and is of class $\mathcal{C}^2$ with respect to $y$, its 
 first derivative with respect to $y$, denoted by $\phi'$ in this paper, is nonnegative for all 
 $x\in \Omega$ and $y\in \mathbb{R}$. In the following, we will omit the first argument of 
 $\phi(\cdot,y)$ and denote it by $\phi(y)$. For all $M>0$, we assume that there exists $C_M>0$
such that
\begin{eqnarray}
|\phi''(y_1)-\phi''(y_2)|\leq C_M|y_1-y_2|\nonumber
\end{eqnarray}
for all $(y_1,y_2)\in [-M,M]^2$.

It is well-known that the state equation (\ref{OCP_state_nonlinear}) admits a unique 
solution $y\in H_0^1(\Omega)\cap L^\infty(\Omega)$ under the aforementioned conditions (see \cite{Arada}).
 Moreover, we have $y\in H_0^1(\Omega)\cap H^2(\Omega)$. Then we are able to introduce the 
 control-to-state mapping $G: L^2(\Omega)\rightarrow H_0^1(\Omega)\cap L^\infty(\Omega)$, 
 which leads to the reduced optimization problem
\begin{eqnarray}
\min\limits_{u\in U_{ad}}\ \ \hat J(u):= J(G(u),u).\label{reduced_semilinear}
\end{eqnarray}
Similar to the linear case, it is easy to prove the existence of a solution to
(\ref{OCP_nonlinear})-(\ref{OCP_state_nonlinear}), see, e.g., \cite{Arada}.
However, the uniqueness is generally not guaranteed. We can also derive the
 first order necessary optimality conditions as
\begin{eqnarray}\label{OCP_OPT_nonlinear}
\hat J'(u)(v-u)=(\alpha u+p,v-u)\geq 0,\ \ \ \forall v\in U_{ad},
\end{eqnarray}
 where the adjoint state $p\in H_0^1(\Omega)$ satisfies
\begin{equation}\label{adjoint_nonlinear}
\left\{\begin{array}{llr}
-\Delta p+\phi'(y)p=y-y_d \ \ &\mbox{in}\
\Omega,\\
 \ \ \ \ \ \ \ \ \ \ \ \ \ \ \ \ \ p=0  \ \ \ &\mbox{on}\ \partial\Omega.
\end{array} \right.
\end{equation}

Moreover, we assume the following second order sufficient optimality condition.
\begin{Assumption}\label{Ass:ssc}
Let $u\in U_{ad}$ fulfil the first order necessary optimality conditions (\ref{OCP_OPT_nonlinear}). 
We assume that there exists a constant $\gamma >0$ such that
\begin{eqnarray}
\hat J''(u)(v,v)\geq \gamma \|v\|_{0,\Omega}^2,\ \ \ \ \forall v\in L^2(\Omega).\nonumber
\end{eqnarray}
\end{Assumption}
We note that Assumption \ref{Ass:ssc} is a rather strong second order sufficient optimality 
condition compared to the one presented in \cite{Arada}, it is commonly used in the error  
estimates of nonlinear optimal control problems (see \cite{Kroner} and \cite{Neitzel}). 
For $u\in U_{ad}$ and $v_1,v_2\in L^2(\Omega)$, the second order derivative of $\hat J$ 
is given by (see \cite{Arada} and \cite{Hinze09book})
\begin{eqnarray}
\hat J''(u)(v_1,v_2)=\int_\Omega(\alpha v_1v_2+ \tilde y_1\tilde y_2-p\phi''(y)\tilde y_1\tilde y_2)dx,\nonumber
\end{eqnarray}
where $y=G(u)$, $\tilde y_i=G'(u)v_i$, $i=1,2$. Now we can show that the second order derivative of 
$\hat J$ is Lipschitz continuous in $L^2(\Omega)$ .
\begin{Lemma}\label{La:4.2}
There exists a constant $C$ such that for all $u_1,u_2\in U_{ad}$ and $v\in L^2(\Omega)$
\begin{eqnarray}
|\hat J''(u_1)(v,v)-\hat J'' (u_2)(v,v)|\leq C\|u_1-u_2\|_{0,\Omega}\|v\|_{0,\Omega}^2\nonumber
\end{eqnarray}
holds.
\end{Lemma}
\begin{proof}
Let $y_i=G(u_i)$, $\tilde y_i=G'(u_i)v$, $i=1,2$, $p_1$ be the adjoint state associated with $u_1$ and
$p_2$ be the adjoint state associated with $u_2$. Then from the definition of the second
order derivative of $\hat J$ we have
\begin{eqnarray*}
|\hat J''(u_1)(v,v)-\hat J'' (u_2)(v,v)|&=&\Big|\int_\Omega(\tilde y_1^2-\tilde y_2^2+p_2\phi''(y_2)
\tilde y_2^2-p_1\phi''(y_1)\tilde y_1^2)dx\Big|\nonumber\\
&\leq&\int_\Omega|(\tilde y_1+\tilde y_2)(\tilde y_1-\tilde y_2)+(p_2-p_1)\phi''(y_1)\tilde y_1^2\nonumber\\
&&-p_2\phi''(y_2)(\tilde y_1^2-\tilde y_2^2)-p_2(\phi''(y_1)-\phi''(y_2))\tilde y_1^2|dx. 
\end{eqnarray*}
This gives
\begin{eqnarray}
&&|\hat J''(u_1)(v,v)-\hat J'' (u_2)(v,v)|\nonumber\\
&\leq&(\|\tilde y_1\|_{0,\Omega}+\|\tilde y_2\|_{0,\Omega})\|\tilde y_1-\tilde y_2\|_{0,\Omega}
+c\|\phi''(y_1)\|_{0,\infty,\Omega}\|p_2-p_1\|_{0,\Omega}\|\tilde y_1\|^2_{0,4,\Omega}\nonumber\\
&&+c\|p_2\|_{0,\infty,\Omega}\big(\|\phi''(y_2)\|_{0,\infty,\Omega}(\|\tilde y_1\|_{0,\Omega}
+\|\tilde y_2\|_{0,\Omega})\|\tilde y_1-\tilde y_2\|_{0,\Omega}+\|y_1-y_2\|_{0,\Omega}\|\tilde y_1\|^2_{0,4,\Omega}\big).\nonumber
\end{eqnarray}
It has been proved in \cite{Arada} that
\begin{eqnarray}
\|G(u)\|_{1,\Omega}\leq C\|u\|_{0,\Omega},\nonumber\\
\|G(u_1)-G(u_2)\|_{0,\Omega}\leq C\|u_1-u_2\|_{0,\Omega},\nonumber\\
\|G'(u_1)v-G'(u_2)v\|_{0,\Omega}\leq C\|u_1-u_2\|_{0,\Omega}\|v\|_{0,\Omega}.\nonumber
\end{eqnarray}
Using the boundedness of $U_{ad}$ and $\phi''(\cdot)$, the embedding
 $H^1(\Omega)\hookrightarrow L^4(\Omega)$, we can obtain the desired result.
 \end{proof}
\begin{Lemma}\label{La:4.3}
Let $u$ satisfy Assumption \ref{Ass:ssc}. Then there exists an $\epsilon>0$ such that
\begin{eqnarray}
\hat J''(w)(v,v)\geq \frac{\gamma}{2}\|v\|_{0,\Omega}^2\nonumber
\end{eqnarray}
holds for all $v\in L^2(\Omega)$ and $w\in U_{ad}$ with $\|w-u\|_{0,\Omega}\leq\epsilon$.
\end{Lemma}
\begin{proof}
From Assumption \ref{Ass:ssc} and Lemma \ref{La:4.2} we have
\begin{eqnarray}
\hat J''(w)(v,v)&=&\hat J''(u)(v,v)+\hat J''(w)(v,v)-\hat J''(u)(v,v)\nonumber\\
&\geq&\gamma\|v\|_{0,\Omega}^2-c\|w-u\|_{0,\Omega}\|v\|_{0,\Omega}^2\nonumber\\
&\geq&{\gamma\over 2}\|v\|_{0,\Omega}^2\nonumber
\end{eqnarray}
with $\epsilon <{{\gamma}\over{2c} }$. This gives the result.
\end{proof}

With this estimate at hand we can prove the local convexity of the objective functional.
\begin{Lemma}\label{La:4.4}
Let $u\in U_{ad}$ satisfy the first order necessary optimality condition (\ref{OCP_OPT_nonlinear}) 
and Assumption \ref{Ass:ssc}. Then there exist constants $\epsilon >0$ and $\gamma>0$ such
 that for all $v\in U_{ad}$ and $w\in U_{ad}$ satisfying $\|v-u\|_{0,\Omega}\leq \epsilon$ 
 and  $\|w-u\|_{0,\Omega}\leq \epsilon$, there holds
\begin{eqnarray}
{\gamma\over 2}\|v-w\|_{0,\Omega}^2\leq (\hat J'(v)-\hat J'(w),v-w).\nonumber
\end{eqnarray}
\end{Lemma}
\begin{proof}
We can conclude from Lemma \ref{La:4.3} that for some $\theta\in [0,1]$
\begin{eqnarray}
(\hat J'(v)-\hat J'(w),v-w)&=&\hat J''(\theta v+(1-\theta)w)(v-w,v-w)\nonumber\\
&\geq &{\gamma\over 2}\|v-w\|_{0,\Omega}^2,\nonumber
\end{eqnarray}
this gives the desired result.
\end{proof}

Now we are ready to define the finite dimensional approximation to the optimal
 control problem (\ref{OCP_nonlinear})-(\ref{OCP_state_nonlinear}):
\begin{eqnarray}\label{OCP_nonlinear_h}
\min\limits_{\bar u_h\in U_{ad}}J_h(\bar y_h,\bar u_h)={1\over
2}\|\bar y_h-y_d\|_{0,\Omega}^2 +
\frac{\alpha}{2}\|\bar u_h\|_{0,\Omega}^2
\end{eqnarray}
subject to
\begin{eqnarray}
a(\bar y_h,v_h)+(\phi(\bar y_h),v_h)=(f+\bar u_h,v_h),\ \ \ \ \forall v_h\in V_h.
\label{OCP_state_nonlinear_h}
\end{eqnarray}
Similar to the continuous case, we can define a discrete control-to-state mapping 
$G_h:L^2(\Omega)\rightarrow V_h$ and formulate a reduced discretised optimization problem
\begin{eqnarray}
\min\limits_{u_h\in U_{ad}}\hat J_h(u_h):=J_h(G_h(u_h),u_h).\nonumber
\end{eqnarray}
The above discretised optimization problem admits at least one solution.
The discretised first order necessary optimality condition can be stated as follows:
\begin{equation}\label{OCP_OPT_nonlinear_h}
\left\{\begin{array}{llr}a(\bar y_h,v_h)+(\phi(\bar y_h),v_h)
=(f+\bar u_h,v_h),\ \ &\forall v_h\in V_h,\\
 a(w_h,\bar p_h)+(\phi'(\bar y_h)\bar p_h,v_h)
 =(\bar y_h-y_d,w_h),\ \ &\forall w_h\in V_h,\\
 (\alpha \bar u_h+\bar p_h,v_h-\bar u_h)\geq 0, \ \ &\forall v_h\in U_{ad}.
\end{array} \right.
\end{equation}
Similar to the proof in \cite{Arada} we can prove the following a priori error estimates
\begin{Lemma}\label{La:4.6}
For $u,v\in L^2(\Omega)$, assume that $G(u)\in H_0^1(\Omega)$ and $G_h(u)\in V_h$ be the solutions 
of the continuous and discretised state equation,  $G'(u)v\in H_0^1(\Omega)$ and $G_h'(u)v\in V_h$ 
be the solutions of the continuous and discretised linearized state equation, respectively.
 Then the following error estimates hold
\begin{eqnarray}
\|G(u)-G_h(u)\|_{0,\Omega}+h\|G(u)-G_h(u)\|_{1,\Omega}\leq Ch^2\|u\|_{0,\Omega},\label{non_error_1}\\
\|G'(u)v-G_h'(u)v\|_{0,\Omega}+h\|G'(u)v-G_h'(u)v\|_{1,\Omega}\leq Ch^{2}\|v\|_{0,\Omega}.\label{non_error_2}
\end{eqnarray}
\end{Lemma}
Now we can formulate the following coercivity of the second order derivative of the discrete reduced objective functional
\begin{Lemma}\label{La:4.5}
Let $u$ satisfy Assumption \ref{Ass:ssc}. Then there exists an $\epsilon>0$ such that
\begin{eqnarray}
\hat J_h''(w)(v,v)\geq \gamma\|v\|_{0,\Omega}^2\nonumber
\end{eqnarray}
holds for all $v\in L^2(\Omega)$ and $w\in U_{ad}$ with $\|w-u\|_{0,\Omega}\leq\epsilon$.
\end{Lemma}
\begin{proof}
Let $y=G(w)$, $y_h=G_h(w)$, $\tilde y = G'(w)v$ and $\tilde y_h=G_h'(w)v$, $p$ and $p_h$ be the
 continuous and discrete adjoint states associated with $w$, respectively. Similar to the proof 
 of Lemma \ref{La:4.2}, using the explicit representations of $\hat J$ and $\hat J_h$ we have
\begin{eqnarray}
|\hat J''(w)(v,v)-\hat J_h'' (w)(v,v)|&=&\Big|\int_\Omega(\tilde y^2-\tilde y_h^2
+p_h\phi''(y_h)\tilde y_h^2-p\phi''(y)\tilde y^2)dx\Big|\nonumber\\
&\leq&\int_\Omega|(\tilde y+\tilde y_h)(\tilde y-\tilde y_h)+(p_h-p)\phi''(y)\tilde y^2\nonumber\\
&&-p_h\phi''(y_h)(\tilde y^2-\tilde y_h^2)-p_h(\phi''(y)-\phi''(y_h))\tilde y^2|dx,\nonumber
\end{eqnarray}
this together with Lemma \ref{La:4.6}, the boundedness of $\phi''(\cdot)$ and the embedding
 $H^1(\Omega)\hookrightarrow L^4(\Omega)$ gives
\begin{eqnarray}\label{Estimate_Semilinear_1}
&&|\hat J''(w)(v,v)-\hat J_h'' (w)(v,v)|\nonumber\\
&\leq&(\|\tilde y\|_{0,\Omega}+\|\tilde y_h\|_{0,\Omega})\|\tilde y-\tilde y_h\|_{0,\Omega}
+c\|\phi''(y)\|_{0,\infty,\Omega}\|p_h-p\|_{0,\Omega}\|\tilde y\|^2_{0,4,\Omega}\nonumber\\
&&+c\|p_h\|_{0,\infty,\Omega}\big(\|\phi''(y_h)\|_{0,\infty,\Omega}(\|\tilde y\|_{0,\Omega}
+\|\tilde y_h\|_{0,\Omega})\|\tilde y-\tilde y_h\|_{0,\Omega}+\|y-y_h\|_{0,\Omega}\|\tilde y\|^2_{0,4,\Omega}\big)\nonumber\\
&\leq&Ch^2\|v\|_{0,\Omega}^2\nonumber\\
&\leq&{\gamma\over 2}\|v\|_{0,\Omega}^2
\end{eqnarray}
for sufficiently small $h$. Combining (\ref{Estimate_Semilinear_1}) with Lemma \ref{La:4.3} we complete the proof.
\end{proof}

Now we are in the position to derive the a priori error estimates for the above finite element approximations
\begin{Theorem} \label{Thm:priori_semilinear}
Let $(u,y,p)\in U_{ad}\times H_0^1(\Omega)\times H_0^1(\Omega)$ and
 $(\bar u_h,\bar y_h,\bar p_h)\in U_{ad}\times V_h\times V_h$ be the
 solutions of problems (\ref{OCP_nonlinear})-(\ref{OCP_state_nonlinear}) and
 (\ref{OCP_OPT_nonlinear_h}), respectively. Then the following error estimates hold
\begin{eqnarray}\label{u_error_nonlinear}
\|u-\bar u_h\|_{0,\Omega}+\|y-\bar y_h\|_{0,\Omega}
+\|p-\bar p_h\|_{0,\Omega}\leq Ch^{2}.
\end{eqnarray}
\end{Theorem}
\begin{proof}
At first, from  Proposition 4.3 and Theorem 4.4 in \cite{Arada} one can prove that $\bar u_h$ 
converges strongly to $u$.  Then, from Lemma \ref{La:4.4} we have
\begin{eqnarray}
{\gamma\over 2}\|u-\bar u_h\|_{0,\Omega}^2&\leq& (\hat J'(u)-\hat J'(\bar u_h), u-\bar u_h)\nonumber\\
&=&(\alpha u+p,u-\bar u_h)-(\alpha \bar u_h+ p(\bar u_h),u-\bar u_h)\nonumber\\
&\leq&(\alpha \bar u_h+ \bar p_h,\bar u_h-u)+(p(\bar u_h)-\bar p_h,\bar u_h-u)\nonumber\\
&\leq&C(\gamma)\|p(\bar u_h)-\bar p_h\|_{0,\Omega}^2+{\gamma\over 4}\|\bar u_h-u\|_{0,\Omega}^2,\label{Thm4.1-1}
\end{eqnarray}
where $p(\bar u_h)\in H_0^1(\Omega)$ is the solution of the following systems
\begin{eqnarray}
a(y(\bar u_h),v)+(\phi(y(\bar u_h)),v)=(f+\bar u_h,v),\ \ \forall v\in H_0^1(\Omega),\nonumber\\
a(v, p(\bar u_h))+(\phi'(y(\bar u_h))p(\bar u_h),v)=(y(\bar u_h)-y_d,v),\ \ \forall v\in H_0^1(\Omega).\nonumber
\end{eqnarray}
Now it remains to estimate $\|p(\bar u_h)-\bar p_h\|_{0,\Omega}$. We have the splitting
\begin{eqnarray}
\|p(\bar u_h)-\bar p_h\|_{0,\Omega}\leq \|p(\bar u_h)-p(\bar y_h)\|_{0,\Omega}+\|p(\bar y_h)-\bar p_h\|_{0,\Omega}\nonumber
\end{eqnarray}
with $p(\bar y_h)\in H_0^1(\Omega)$ the solution of the following equation
\begin{eqnarray}
a(v, p(\bar y_h))+(\phi'(\bar y_h)p(\bar y_h),v)=(\bar y_h-y_d,v),\ \ \forall v\in H_0^1(\Omega).\nonumber
\end{eqnarray}
Because $\phi'(\cdot)\geq 0$ and $\phi'$ is Lipschitz continuous, setting $v=p(\bar u_h)-p(\bar y_h)$ we have
\begin{eqnarray}
\|p(\bar u_h)-p(\bar y_h)\|_{1,\Omega}^2&\leq& a(p(\bar u_h)-p(\bar y_h),p(\bar u_h)-p(\bar y_h))\nonumber\\
&=&(y(\bar u_h)-\bar y_h,v)-(\phi'(y(\bar u_h))p(\bar u_h)-\phi'(\bar y_h)p(\bar y_h),v)\nonumber\\
&=&(y(\bar u_h)-\bar y_h,v)-((\phi'(y(\bar u_h))-\phi'(\bar y_h))p(\bar u_h),v)\nonumber\\
&&-(\phi'(\bar y_h)(p(\bar u_h)-p(\bar y_h)),v)\nonumber\\
&\leq&(y(\bar u_h)-\bar y_h,v)-((\phi'(y(\bar u_h))-\phi'(\bar y_h))p(\bar u_h),v)\nonumber\\
&\leq&\|y(\bar u_h)-\bar y_h\|_{0,\Omega}\|v\|_{0,\Omega}+C\|y(\bar u_h)-\bar y_h\|_{0,\Omega}
\|p(\bar u_h)\|_{0,4,\Omega}\|v\|_{0,4,\Omega}\nonumber\\
&\leq&C\|y(\bar u_h)-\bar y_h\|_{0,\Omega}(1+\|p(\bar u_h)\|_{1,\Omega})\|v\|_{1,\Omega},\label{Thm4.1-2_new}
\end{eqnarray}
where we used the Sobolev embedding theorem in the last inequality. Collecting the above estimates we arrive at
\begin{eqnarray}
\|u-\bar u_h\|_{0,\Omega}\leq C\|y(\bar u_h)-\bar y_h\|_{0,\Omega}+C\|p(\bar y_h)-\bar p_h\|_{0,\Omega}.\label{Thm4.1-3}
\end{eqnarray}
Note that $\bar y_h$ and $\bar p_h$ are the standard finite element approximations of $y(\bar u_h)$ and $p(\bar y_h)$, respectively. 
Standard error estimates (cf. \cite{Arada}) yield
\begin{eqnarray}
\|u-\bar u_h\|_{0,\Omega}\leq Ch^{2}.\label{Thm4.1-5}
\end{eqnarray}
Similar to (\ref{Thm4.1-2_new}) we can prove that
\begin{eqnarray}
\|y-y(\bar u_h)\|_{0,\Omega}+\|p-p(\bar u_h)\|_{0,\Omega}\leq C\|u-\bar u_h\|_{0,\Omega}.\nonumber
\end{eqnarray}
Then triangle inequality implies that
\begin{eqnarray}
\|y-\bar y_h\|_{0,\Omega}+\|p-\bar p_h\|_{0,\Omega}\leq Ch^{2}.\label{Thm4.1-6}
\end{eqnarray}
This completes the proof.
\end{proof}

Assume that we have obtained the approximate solution
$(u_{h_{k}},y_{h_{k}},p_{h_{k}})\in U_{ad}\times V_{H,h_{k}}\times V_{H,h_{k}}$ on the $k$-th
level mesh $\mathcal{T}_{h_k}$. Now we introduce an one correction step to improve the accuracy of
 the current approximation $(u_{h_{k}},y_{h_{k}},p_{h_{k}})$.
\begin{Algorithm}\label{Alg:4.1}One correction step:
\begin{enumerate}
\item Find $y_{h_{k+1}}^*\in V_{h_{k+1}}$ such that
\begin{equation}\label{step_2_semilinear}
a(y^*_{h_{k+1}},v_{h_{k+1}})
= (f+u_{h_{k}},v_{h_{k+1}})-(\phi(y_{h_k}),v_{h_{k+1}}),\ \ \ \ \forall v_{h_{k+1}}\in V_{h_{k+1}}.
\end{equation}
Solve the above equation with multigrid method to obtain an approximation
$\hat y_{h_{k+1}}\in V_{h_{k+1}}$ with error
$\|\hat y_{h_{k+1}}-y_{h_{k+1}}^*\|_{1,\Omega}\leq Ch_{h_{k}}^{2}$
and define $\hat y_{h_{k+1}}:=MG(V_{h_{k+1}},u_{h_k})$.

\item Find $p_{h_{k+1}}^*\in V_{h_{k+1}}$ such that
\begin{equation}\label{step_3_semilinear}
a(w_{h_{k+1}},p^*_{h_{k+1}})+(\phi'(\hat y_{h_{k+1}})p^*_{h_{k+1}},v_{h_{k+1}})
 = (\hat y_{h_{k+1}}-y_d,v_{h_{k+1}}),\ \ \ \ \forall v_{h_{k+1}}\in V_{h_{k+1}}.
\end{equation}
Solve the above equation with multigrid method to obtain an approximation
$\hat p_{h_{k+1}}\in V_{h_{k+1}}$ with error
$\|\hat p_{h_{k+1}}-p_{h_{k+1}}^*\|_{1,\Omega}\leq Ch_{h_{k}}^{2}$
and define $\hat p_{h_{k+1}}:=MG(V_{h_{k+1}},\hat y_{h_{k+1}})$.

\item Define a new finite element space
$V_{H,h_{k+1}}:=V_H+{\rm span}\{\hat y_{h_{k+1}}\}+{\rm span}\{\hat p_{h_{k+1}}\}$
and solve the following optimal control problem:
\begin{eqnarray}\label{step_4_semilinear}
\min\limits_{u_{h_{k+1}}\in U_{ad},\ y_{h_{k+1}}
\in V_{H,h_{k+1}}} J(y_{h_{k+1}},u_{h_{k+1}})&={1\over
2}\|y_{h_{k+1}}-y_d\|_{0,\Omega}^2 +
\frac{\alpha}{2}\|u_{h_{k+1}}\|_{0,\Omega}^2
\end{eqnarray}
subject to
\begin{equation}\label{step_4_state_semilinear}
a(y_{h_{k+1}},v_{H,h_{k+1}})+(\phi(y_{h_{k+1}}),v_{H,h_{k+1}})
= (f+u_{h_{k+1}},v_{H,h_{k+1}}),\ \ \ \ \forall v_{H,h_{k+1}}\in V_{H,h_{k+1}}.
\end{equation}
The corresponding optimality condition reads:
 Find $(u_{h_{k+1}},y_{h_{k+1}},p_{h_{k+1}})\in U_{ad}
 \times V_{H,h_{k+1}}\times V_{H,h_{k+1}}$ such that
\begin{equation}\label{step_4_OPT_semilinear}
\left\{\begin{array}{lll}
a(y_{h_{k+1}},v_{H,h_{k+1}}) +(\phi(y_{h_{k+1}}),v_{H,h_{k+1}})
= (f+u_{h_{k+1}},v_{H,h_{k+1}}), &\forall v_{H,h_{k+1}}\in V_{H,h_{k+1}},\\
a(v_{H,h_{k+1}},p_{h_{k+1}}) +(\phi'(y_{h_{k+1}})p_{h_{k+1}},v_{H,h_{k+1}})
= (y_{h_{k+1}}-y_d,v_{H,h_{k+1}}), &\forall v_{H,h_{k+1}}\in V_{H,h_{k+1}},\\
(\alpha u_{h_{k+1}}+p_{h_{k+1}},v-u_{h_{k+1}})\geq 0, &\forall v\in U_{ad}.
\end{array} \right.
\end{equation}
\end{enumerate}
We define the output of above algorithm as
\begin{eqnarray}\label{correction_semilinear}
(u_{h_{k+1}},y_{h_{k+1}},p_{h_{k+1}})=\mbox{\rm Correction}(V_H,u_{h_{k}},y_{h_k}, p_{h_k}, V_{h_{k+1}}).
\end{eqnarray}
\end{Algorithm}
\begin{Remark}
As for the linear case, in Algorithm \ref{Alg:4.1} one needs to solve a nonlinear optimization 
problem (\ref{step_4_semilinear})-(\ref{step_4_state_semilinear}) on the finite element space
$V_{H,h_{k+1}}$, there are several ways to provide a good initial guess for the optimization 
algorithm which may speed up the convergence. Also, the good initial guess would lead to the 
correct solution for the nonlinear optimization problem. One option is to use $u_{h_{k}}$ as 
initial guess, while the other choice is $P_{U_{ad}}\{-\frac{1}{\alpha}\hat p_{h_{k+1}}\}$.
\end{Remark}
In the following of this paper, we denote
$(\bar u_{h_{k+1}},\bar y_{h_{k+1}},\bar p_{h_{k+1}})\in U_{ad}
\times V_{h_{k+1}}\times V_{h_{k+1}}$ the finite element solution to the
discrete optimal control problems (\ref{OCP_nonlinear_h})-(\ref{OCP_state_nonlinear_h})
in the finite element space $V_{h_{k+1}}$.
We are able to analyze the error estimates between solutions
 $(\bar u_{h_{k+1}},\bar y_{h_{k+1}},\bar p_{h_{k+1}})$
 and the correction one $(u_{h_{k+1}},y_{h_{k+1}},p_{h_{k+1}})$
 on mesh level $\mathcal{T}_{h_{k+1}}$.
\begin{Theorem}\label{Thm:4.1}
Let $(\bar u_{h_{k+1}},\bar y_{h_{k+1}},
\bar p_{h_{k+1}})\in U_{ad}\times V_{h_{k+1}}\times V_{h_{k+1}}$ be
the solution of problems (\ref{OCP_nonlinear_h})-(\ref{OCP_state_nonlinear_h}) and
$(u_{h_{k+1}},y_{h_{k+1}},p_{h_{k+1}})$ be the numerical approximation
by Algorithm \ref{Alg:4.1}, respectively.
Assume there exists a real number $\eta_{h_k}$ such that
$(u_{h_k}, y_{h_k}, p_{h_k})$ have the following error estimates
\begin{eqnarray}\label{assume_error_nonlinear}
\|\bar u_{h_k}-u_{h_k}\|_{0,\Omega}+\|\bar y_{h_k}-y_{h_k}\|_{0,\Omega}
+\|\bar p_{h_k}-p_{h_k}\|_{0,\Omega} = \eta_{h_k}.
\end{eqnarray}
Then the following error estimates hold
\begin{eqnarray}\label{correct_error_nonlinear}
\|\bar u_{h_{k+1}}-u_{h_{k+1}}\|_{0,\Omega}
+\|\bar y_{h_{k+1}}-y_{h_{k+1}}\|_{0,\Omega}
+\|\bar p_{h_{k+1}}-p_{h_{k+1}}\|_{0,\Omega}
\leq C\eta_{h_{k+1}},
\end{eqnarray}
where $\eta_{h_{k+1}}=H(h_{k}^{2}+\eta_{h_{k}})$.
\end{Theorem}
\begin{proof}
Setting $v=\bar y_{h_{k+1}}-y_{h_{k+1}}^*$, from the state equation approximation we have
\begin{eqnarray}
\|\bar y_{h_{k+1}}-y_{h_{k+1}}^*\|_{1,\Omega}^2&\leq& a(\bar y_{h_{k+1}}-y_{h_{k+1}}^*,\bar y_{h_{k+1}}-y_{h_{k+1}}^*)\nonumber\\
&=&(\bar u_{h_{k+1}}-u_{h_k},v)-(\phi(\bar y_{h_{k+1}})-\phi(y_{h_k}),v)\nonumber\\
&\leq&C(\|\bar u_{h_{k+1}}-u_{h_k}\|_{0,\Omega}+\|\bar y_{h_{k+1}}-y_{h_{k}}\|_{0,\Omega})\|v\|_{1,\Omega},\nonumber
\end{eqnarray}
which together with Theorem \ref{Thm:priori_semilinear} implies that
\begin{eqnarray}
\|\bar y_{h_{k+1}}-y_{h_{k+1}}^*\|_{1,\Omega}&\leq& C\|\bar u_{h_{k+1}}-\bar u_{h_k}\|_{0,\Omega}+C\|\bar u_{h_{k}}-u_{h_k}\|_{0,\Omega}\nonumber\\
&&+C\|\bar y_{h_{k+1}}-\bar y_{h_k}\|_{0,\Omega}+C\|\bar y_{h_{k}}-y_{h_k}\|_{0,\Omega}\nonumber\\
&\leq&C(h_k^{2}+\eta_k).\label{Thm4.1-1}
\end{eqnarray}
So we can derive
\begin{eqnarray}
\|\bar y_{h_{k+1}}-\hat y_{h_{k+1}}\|_{1,\Omega}&\leq&\|\bar y_{h_{k+1}}-y_{h_{k+1}}^*\|_{1,\Omega}+\|y_{h_{k+1}}^*-\hat y_{h_{k+1}}\|_{1,\Omega}\nonumber\\
&\leq&C(h_k^{2}+\eta_k).\label{Thm4.1-2}
\end{eqnarray}
Setting $v=\bar p_{h_{k+1}}-p_{h_{k+1}}^*$, we conclude from the adjoint state equation and $\phi'(\cdot)\geq 0$ that
\begin{eqnarray}
\|\bar p_{h_{k+1}}-p_{h_{k+1}}^*\|_{1,\Omega}^2&\leq& a(\bar p_{h_{k+1}}-p_{h_{k+1}}^*,\bar p_{h_{k+1}}-p_{h_{k+1}}^*)\nonumber\\
&=&(\bar y_{h_{k+1}}-\hat y_{h_{k+1}},v)-(\phi'(\bar y_{h_{k+1}})\bar p_{h_{k+1}}-\phi'(\hat y_{h_{k+1}})p_{h_{k+1}}^*,v)\nonumber\\
&=&(\bar y_{h_{k+1}}-\hat y_{h_{k+1}},v)-((\phi'(\bar y_{h_{k+1}})-\phi'(\hat y_{h_{k+1}}))\bar p_{h_{k+1}},v)\nonumber\\
&&-(\phi'(\hat y_{h_{k+1}})(\bar p_{h_{k+1}}-p_{h_{k+1}}^*),v)\nonumber\\
&\leq&(\bar y_{h_{k+1}}-\hat y_{h_{k+1}},v)-((\phi'(\bar y_{h_{k+1}})-\phi'(\hat y_{h_{k+1}}))\bar p_{h_{k+1}},v)\nonumber\\
&\leq&C\|\bar y_{h_{k+1}}-\hat y_{h_{k+1}}\|_{0,\Omega}\|v\|_{0,\Omega}
+\|\bar y_{h_{k+1}}-\hat y_{h_{k+1}}\|_{0,\Omega}\|\bar p_{h_{k+1}}\|_{0,4,\Omega}\|v\|_{0,4,\Omega}\nonumber\\
&\leq&C\|\bar y_{h_{k+1}}-\hat y_{h_{k+1}}\|_{0,\Omega}(1+\|\bar p_{h_{k+1}}\|_{1,\Omega})\|v\|_{1,\Omega},\nonumber
\end{eqnarray}
which gives
\begin{eqnarray}
\|\bar p_{h_{k+1}}-p_{h_{k+1}}^*\|_{1,\Omega}\leq C\|\bar y_{h_{k+1}}-\hat y_{h_{k+1}}\|_{0,\Omega}
\leq C(h_k^{2}+\eta_k).\label{Thm4.1-3}
\end{eqnarray}
Similar to (\ref{Thm4.1-2}) we have
\begin{eqnarray}
\|\bar p_{h_{k+1}}-\hat p_{h_{k+1}}\|_{1,\Omega}&\leq&\|\bar p_{h_{k+1}}-p_{h_{k+1}}^*\|_{1,\Omega}
+\|p_{h_{k+1}}^*-\hat p_{h_{k+1}}\|_{1,\Omega}\nonumber\\
&\leq&C(h_k^{2}+\eta_k).\label{Thm4.1-4}
\end{eqnarray}
From the coercivity of the second order derivative of the discrete reduced objective functional presented 
in Lemma \ref{La:4.4}, for some $\theta\in [0,1]$ we can derive
\begin{eqnarray}
&&\gamma\|\bar u_{h_{k+1}}-u_{h_{k+1}}\|_{0,\Omega}^2\nonumber\\
&\leq& \hat J_h''(\theta \bar u_{h_{k+1}}+(1-\theta)u_{h_{k+1}})
(\bar u_{h_{k+1}}-u_{h_{k+1}},\bar u_{h_{k+1}}-u_{h_{k+1}})\nonumber\\
&=&\hat J_h'(\bar u_{h_{k+1}})(\bar u_{h_{k+1}}-u_{h_{k+1}})-\hat J_h'(u_{h_{k+1}})(\bar u_{h_{k+1}}-u_{h_{k+1}})\nonumber\\
&=&(\alpha \bar u_{h_{k+1}}+p_{h_{k+1}}(\bar u_{h_{k+1}}),\bar u_{h_{k+1}}-u_{h_{k+1}})
-(\alpha u_{h_{k+1}}+p_{h_{k+1}},\bar u_{h_{k+1}}-u_{h_{k+1}})\nonumber\\
&\leq&(\alpha \bar u_{h_{k+1}}+\bar p_{h_{k+1}},\bar u_{h_{k+1}}-u_{h_{k+1}})
+(p_{h_{k+1}}(\bar u_{h_{k+1}})-\bar p_{h_{k+1}},\bar u_{h_{k+1}}-u_{h_{k+1}})\nonumber\\
&\leq&(p_{h_{k+1}}(\bar u_{h_{k+1}})-\bar p_{h_{k+1}},\bar u_{h_{k+1}}-u_{h_{k+1}}),\nonumber
\end{eqnarray}
which implies that
\begin{eqnarray}
&&{\gamma\over 4}\|\bar u_{h_{k+1}}-u_{h_{k+1}}\|_{0,\Omega}
\leq C\|p_{h_{k+1}}(\bar u_{h_{k+1}})-\bar p_{h_{k+1}}\|_{0,\Omega}\nonumber\\
&\leq&C\|p_{h_{k+1}}(\bar u_{h_{k+1}})-p_{h_{k+1}}(\bar y_{h_{k+1}})\|_{0,\Omega}
+C\|p_{h_{k+1}}(\bar y_{h_{k+1}})-\bar p_{h_{k+1}}\|_{0,\Omega}\nonumber\\
&\leq&C\|y_{h_{k+1}}(\bar u_{h_{k+1}})-\bar y_{h_{k+1}}\|_{0,\Omega}
+C\|p_{h_{k+1}}(\bar y_{h_{k+1}})-\bar p_{h_{k+1}}\|_{0,\Omega}.\label{Thm4.1-5}
\end{eqnarray}
It is easy to see that $y_{h_{k+1}}(\bar u_{h_{k+1}})$ is the finite element
 approximation to $\bar y_{h_{k+1}}$ on $V_{H,h_{k+1}}$ for the semilinear elliptic 
 equation because of $V_{H,h_{k+1}}\subset V_{h_{k+1}}$. A Ce\'{a}-lemma for semilinear elliptic equation implies
\begin{eqnarray}
\|y_{h_{k+1}}(\bar u_{h_{k+1}})-\bar y_{h_{k+1}}\|_{1,\Omega}
&\leq& C\inf\limits_{v_{H,h_{k+1}}\in V_{H,h_{k+1}}}\|\bar y_{h_{k+1}}-v_{H,h_{k+1}}\|_{1,\Omega}\nonumber\\
&\leq&C\|\bar y_{h_{k+1}}-\hat y_{h_{k+1}}\|_{1,\Omega}\nonumber\\
&\leq& C(h_k^{2}+\eta_{h_k}).\label{Thm4.1-6}
\end{eqnarray}
Now we prove the improved $L^2$-norm estimate by Aubin-Nitsche argument. Consider the following adjoint equation
\begin{equation}\label{duality}
\left\{\begin{array}{llr} -\Delta \psi+\phi'(\bar y_{h_{k+1}})\psi
= y_{h_{k+1}}(\bar u_{h_{k+1}})-\bar y_{h_{k+1}}\ \ &\mbox{in}\
\Omega, \\
\ \ \ \ \ \ \ \ \ \ \ \ \ \ \ \ \ \ \ \ \ \  \psi=0  \ \ \ &\mbox{on}\ \partial\Omega.
\end{array}
\right.
\end{equation}
Then we have $\|\psi\|_{2,\Omega}\leq C\|y_{h_{k+1}}(\bar u_{h_{k+1}})-\bar y_{h_{k+1}}\|_{0,\Omega}$.
 Setting $v=y_{h_{k+1}}(\bar u_{h_{k+1}})-\bar y_{h_{k+1}}$ we have
\begin{eqnarray}
&&\|y_{h_{k+1}}(\bar u_{h_{k+1}})-\bar y_{h_{k+1}}\|_{0,\Omega}^2
=(y_{h_{k+1}}(\bar u_{h_{k+1}})-\bar y_{h_{k+1}},v)\nonumber\\
&=&a(y_{h_{k+1}}(\bar u_{h_{k+1}})-\bar y_{h_{k+1}},\psi)
+(\phi'(\bar y_{h_{k+1}})\psi,y_{h_{k+1}}(\bar u_{h_{k+1}})-\bar y_{h_{k+1}})\nonumber\\
&=&a(y_{h_{k+1}}(\bar u_{h_{k+1}})-\bar y_{h_{k+1}},\psi-\Pi_H\psi)
+(\phi(y_{h_{k+1}}(\bar u_{h_{k+1}}))-\phi(\bar y_{h_{k+1}}),\psi-\Pi_H\psi)\nonumber\\
&&+(\phi'(\bar y_{h_{k+1}})\psi,y_{h_{k+1}}(\bar u_{h_{k+1}})-\bar y_{h_{k+1}})
+((\bar y_{h_{k+1}}-y_{h_{k+1}}(\bar u_{h_{k+1}}))\phi'(\theta)),\psi)\nonumber\\
&\leq&CH\|\psi\|_{2,\Omega}\|y_{h_{k+1}}(\bar u_{h_{k+1}})-\bar y_{h_{k+1}}\|_{1,\Omega}
+C\|\psi\|_{0,\infty,\Omega}\|\phi''(\xi)\|_{0,\infty,\Omega}\|y_{h_{k+1}}
(\bar u_{h_{k+1}})-\bar y_{h_{k+1}}\|_{0,\Omega}^2\nonumber\\
&\leq&C\|\psi\|_{2,\Omega}(H\|y_{h_{k+1}}(\bar u_{h_{k+1}})-\bar y_{h_{k+1}}\|_{1,\Omega}
+\|y_{h_{k+1}}(\bar u_{h_{k+1}})-\bar y_{h_{k+1}}\|_{0,\Omega}^2),\nonumber
\end{eqnarray}
where $\Pi_H$ denotes the interpolation of $\psi$ in the finite element space $V_H$, 
$\theta=a_1\bar y_{h_{k+1}}+(1-a_1)y_{h_{k+1}}(\bar u_{h_{k+1}})$ for some $a_1\in [0,1]$ 
and $\xi=a_2\bar y_{h_{k+1}}+(1-a_2)\theta$ for some $a_2\in [0,1]$. 
Since $\|y_{h_{k+1}}(\bar u_{h_{k+1}})-\bar y_{h_{k+1}}\|_{0,\Omega}\ll 1$, we can conclude that
\begin{eqnarray}\label{Thm4.1-7}
\|y_{h_{k+1}}(\bar u_{h_{k+1}})-\bar y_{h_{k+1}}\|_{0,\Omega}&\leq&
CH\|y_{h_{k+1}}(\bar u_{h_{k+1}})-\bar y_{h_{k+1}}\|_{1,\Omega}\nonumber\\
&\leq&CH(h_k^{2}+\eta_{h_k}).
\end{eqnarray}
Similar to (\ref{Thm4.1-6})-(\ref{Thm4.1-7}), we can derive
\begin{eqnarray}\label{Thm4.1-8}
\|p_{h_{k+1}}(\bar y_{h_{k+1}})-\bar p_{h_{k+1}}\|_{0,\Omega}
&\leq&CH(h_k^{2}+\eta_{h_k}).
\end{eqnarray}
Combining (\ref{Thm4.1-5}), (\ref{Thm4.1-7}) and (\ref{Thm4.1-8}) leads to
the following estimate
\begin{eqnarray}\label{Thm4.1-9}
\|\bar u_{h_{k+1}}-u_{h_{k+1}}\|_{0,\Omega}
&\leq&CH(h_k^{2}+\eta_{h_k}).
\end{eqnarray}
Using the triangle inequality, we obtain
\begin{eqnarray*}
\|\bar u_{h_{k+1}}-u_{h_{k+1}}\|_{0,\Omega}+\|\bar y_{h_{k+1}}-y_{h_{k+1}}\|_{0,\Omega}
+\|\bar p_{h_{k+1}}-p_{h_{k+1}}\|_{0,\Omega}
\leq CH(h_k^{2}+\eta_{h_k}).
\end{eqnarray*}
This is the desired result (\ref{correct_error_nonlinear}) and the proof is complete. 
\end{proof}
Based on the sequence of nested finite element spaces
$V_{h_1}\subset V_{h_2}\subset\cdots\subset V_{h_n}$ and the one correction step
defined in Algorithm \ref{Alg:4.1},
we can define the multilevel correction method to solve the nonlinear optimal control problem:
\begin{Algorithm}\label{Alg:4.2}
A multilevel correction method for nonlinear optimal control problem:
\begin{enumerate}
\item Solve a nonlinear optimal control problem in the initial space $V_{h_1}$:
\begin{eqnarray}\label{step_1_semilinear}
\min\limits_{u_{h_1}\in U_{ad},\ y_{h_1}\in V_{h_1}}\ \ J(y_{h_1},u_{h_1})={1\over
2}\|y_{h_1}-y_d\|_{0,\Omega}^2 +
\frac{\alpha}{2}\|u_{h_1}\|_{0,\Omega}^2
\end{eqnarray}
subject to
\begin{equation}\label{step_1_state_semilinear}
a(y_{h_1},v_{h_1}) +(\phi(y_{h_1}),v_{h_1})= (f+u_{h_1},v_{h_1}),
\ \ \ \ \forall v_{h_1}\in V_{h_1}.
\end{equation}
The corresponding optimality condition reads:
Find $(u_{h_1},y_{h_1},p_{h_1})\in U_{ad}\times V_{h_1}\times V_{h_1}$ such that
\begin{equation}\label{step_1_OPT_semilinear}
\left\{\begin{array}{llr}
a(y_{h_1},v_{h_1})+(\phi(y_{h_1}),v_{h_1})
= (f+u_{h_1},v_{h_1}),\ \ &\forall v_{h_1}\in V_{h_1},\\
a(v_{h_1},p_{h_1}) +(\phi'(y_{h_1})p_{h_1},v_{h_1})
= (y_{h_1}-y_d,v_{h_1}),\ \ &\forall v_{h_1}\in V_{h_1},\\
(\alpha u_{h_1}+p_{h_1},v-u_{h_1})\geq 0,\ \ &\forall v\in U_{ad}.
\end{array}
\right.
\end{equation}

\item Do $k=1$, $\cdots$, $n-1$
	
Obtain a new optimal solution $(u_{h_{k+1}},y_{h_{k+1}},p_{h_{k+1}})\in U_{ad}
\times V_{h_{k+1}}\times V_{h_{k+1}}$ by Algorithm \ref{Alg:4.1}
\begin{eqnarray*}
(u_{h_{k+1}},y_{h_{k+1}},p_{h_{k+1}})={\rm Correction}(V_H,u_{h_{k}},y_{h_k}, p_{h_k}, V_{h_{k+1}}).
\end{eqnarray*}
end Do
\end{enumerate}
Finally, we obtain a numerical approximation $(u_{h_n},y_{h_n}, p_{h_n})
\in U_{ad}\times V_{h_{n}}\times V_{h_n}$ for
 problem (\ref{OCP_nonlinear})-(\ref{OCP_state_nonlinear}).
\end{Algorithm}

Now we are in the position to give the error estimates for the solution generated by the above
 multilevel correction scheme described in Algorithm \ref{Alg:4.2}.
\begin{Theorem}\label{Thm:4.2}
Let $(u,y,p)\in U_{ad}\times H_0^1(\Omega)\times H_0^1(\Omega)$ and
 $(u_{h_n},y_{h_n},p_{h_n})\in U_{ad}\times V_{h_n}\times V_{h_n}$
 be the solution of problems (\ref{OCP_nonlinear})-(\ref{OCP_state_nonlinear})
 and the solution by Algorithm \ref{Alg:4.2}, respectively.
Assume the mesh size $H$ satisfies the condition  $CH\beta^2<1$.
Then the following error estimates hold
\begin{eqnarray}\label{multigrid_error_nonlinear}
\|\bar{u}_{h_n}-u_{h_n}\|_{0,\Omega}+\|\bar{y}_{h_n}-y_{h_n}\|_{0,\Omega}
+\|\bar{p}_{h_n}-p_{h_n}\|_{0,\Omega}\leq Ch_n^{2}.
\end{eqnarray}
Finally, we have the following error estimates
\begin{eqnarray}\label{final_error_nonlinear}
\|u-u_{h_n}\|_{0,\Omega}+\|y-y_{h_n}\|_{0,\Omega}+\|p-p_{h_n}\|_{0,\Omega}\leq
Ch_n^{2}.
\end{eqnarray}
\end{Theorem}
\begin{proof}
The proof is similar to the proof of Theorem \ref{Thm:3.2}, we omit it here.
\end{proof}

\section{Numerical Examples}\setcounter{equation}{0}

To test the efficiency of our proposed algorithm, we in this section carry
out some numerical experiments. All the computations are based on the C++ library AFEPack (see \cite{Li and Liu}). 
At first, we  consider the following linear-quadratic optimal control problem:
\begin{eqnarray}
\min\limits_{u\in U_{ad}} J(y,u)={1\over
2}\|y-y_d\|_{0,\Omega}^2 +
\frac{\alpha}{2}\|u\|_{0,\Omega}^2\nonumber
\end{eqnarray}
subject to
\begin{equation}\label{example_state}
\left\{\begin{array}{llr} -\Delta y=f+u \ \ &\mbox{in}\
\Omega, \\
 \ y=0  \ \ \ &\mbox{on}\ \partial\Omega,\\
 a(x)\leq u(x)\leq b(x)\ \ &\mbox{a.e.\ in}\ \Omega.
\end{array} \right.
\end{equation}

\begin{Example}\label{Exm:1.1}
We set $\Omega=[0,1]^2$. Let
$a=-1$, $b=1$, $\alpha =0.1$, $g(x_1,x_2)=2\pi^2\sin(\pi x_1)\sin(\pi x_2)$.
Then $f$ is chosen as
\begin{eqnarray}\label{nonumber}
f(x_1,x_2) =\left\{\begin{array}{llr}
g(x_1,x_2)-a,\ \ \ \ \ &g(x_1,x_2)< a,\\
0, \  \ \ \ \ &a\leq g(x_1,x_2)\leq b,\\
g(x_1,x_2)-b,\ \ \ \ \ &g(x_1,x_2)>b.
\end{array}\right.
\end{eqnarray}
Due to the state equation (\ref{example_state}),
we obtain the exact optimal control $u$
\begin{eqnarray}\label{nonumber}
u(x_1,x_2) =\left\{\begin{array}{llr}
a,\ \ \ \ \ &g(x_1,x_2)< a;\\
g(x_1,x_2), \  \ \ \ \ &a\leq g(x_1,x_2)\leq b;\\
b,\ \ \ \ \ &g(x_1,x_2)>b.
\end{array}\right.
\end{eqnarray}
We also have
\begin{eqnarray*}
y(x_1,x_2)=\sin(\pi x_1)\sin(\pi x_2),\nonumber\\
p(x_1,x_2)=-2\pi^2\alpha\sin(\pi x_1)\sin(\pi x_2).
\end{eqnarray*}
The desired state is given by
\begin{eqnarray*}
y_d(x_1,x_2)=y(x_1,x_2)+4\pi^4\alpha\sin(\pi x_1)\sin(\pi x_2).
\end{eqnarray*}
\end{Example}
At first, we consider the comparison of errors for the solutions by the direct
solving of optimal control problem and the multilevel correction method defined by
 Algorithm \ref{Alg:3.2}, respectively, on the  sequence of nested linear finite element
spaces $V_{h_1}\subset V_{h_2}\subset V_{h_3}$ which are defined on the three level
meshes $\mathcal{T}_{h_1}$, $\mathcal{T}_{h_2}$ and $\mathcal{T}_{h_3}$.
Here we set $\mathcal{T}_H=\mathcal{T}_{h_1}$.
 The series of meshes $\mathcal{T}_{h_1}$, $\mathcal{T}_{h_2}$ and $\mathcal{T}_{h_3}$
 are produced by regular refinements with $\beta=2$.
The optimal control is discretized implicitly by variational
discretization concept proposed by Hinze \cite{Hinze05COAP} and
the discretized optimization problem is solved by projected gradient method (see \cite{LiuYan08book}).

In Algorithm \ref{Alg:3.2}, we note that on the coarsest finite element
space $V_{h_1}$ one needs to solve the optimization problem directly,
while on finer finite element spaces $V_{h_2}$ and $V_{h_3}$ one only needs to solve two
linear boundary value problems and one optimization problem in the coarsest finite element
 space $V_{h_1}$. From Tables \ref{table:1.1} and \ref{table:1.2}, we can observe
 that with same degree of freedoms the comparable errors can be obtained
 on finer finite element spaces $V_{h_2}$ and $V_{h_3}$ but
 with greatly reduced computational complexity by the multilevel correction method.
\begin{table}[ht]
\centering
\caption{Convergence history of $\|u-u_h\|_{0,\Omega}$ for Example \ref{Exm:1.1}.}\label{table:1.1}
\begin{tabular}{||c|c|c|c|c|c|c|c|c||}
\hline
$\#V_{h_1}$&$\|u-u_{h_1}\|_{0,\Omega}$&order&$\#V_{h_2}$&
$\|u-u_{h_2}\|_{0,\Omega}$&order&$\#V_{h_3}$&$\|u-u_{h_3}\|_{0,\Omega}$&order\\
\hline
$139$& 1.781810e-2& &  &   &  &  &&\\
\hline
$513$& 5.343963e-3 &1.7374   & 513  & 5.343963e-3&  & &&  \\
\hline
$1969$&1.316909e-3&2.0208   &  1969  &1.316909e-3& 2.0208  & 1969 &1.316909e-3& \\
\hline
$7713$&3.289924e-4&2.0010    &   7713  &3.289928e-4& 2.0010 & 7713 &3.289925e-4&2.0010  \\
\hline
$30529$&8.206814e-5&2.0032   &  30529 &8.208219e-5& 2.0029 & 30529 &8.206746e-5&2.0032  \\
\hline
$ 121473$& 2.051887e-5 &1.9999  &  121473 & 2.051903e-5& 2.0001 &  121473 &2.051942e-5&1.9998  \\
\hline
\end{tabular}
\end{table}
\begin{table}[ht]
\centering
\caption{ Convergence history of $\|y-y_h\|_{0,\Omega}$
for Example \ref{Exm:1.1}.}\label{table:1.2}
\begin{tabular}{||c|c|c|c|c|c|c|c|c||}
\hline
$\#V_{h_1}$&$\|y-y_{h_1}\|_{0,\Omega}$&order&$\#V_{h_2}$&
$\|y-y_{h_2}\|_{0,\Omega}$&order&$\#V_{h_3}$&$\|y-y_{h_3}\|_{0,\Omega}$&order\\
\hline
$139$&6.809744e-3& &  &   &  &  &&\\
\hline
$513$&1.722517e-3&1.9831   & 513  & 1.722504e-3&  & &&  \\
\hline
$1969$&4.322249e-4&1.9947   &  1969  &4.322236e-4& 1.9947  & 1969 &4.322212e-4&\\
\hline
$7713$&1.081827e-4&1.9983    &   7713  &1.081826e-4& 1.9983 & 7713 &1.081824e-4&1.9983\\
\hline
$30529$&2.705474e-5&1.9995   &  30529 &2.705475e-5& 1.9995& 30529 &2.705475e-5&1.9995\\
\hline
$121473 $& 6.764356e-6 &1.9999  &  121473 &6.764348e-6& 1.9999 &  121473 &6.764373e-6&1.9999\\
\hline
\end{tabular}
\end{table}

Then we test our proposed multilevel correction algorithm on the sequence of
multilevel meshes.
Two initial meshes with $68$ and $139$ nodes as shown in Figure \ref{fig:1} are used.
We show the errors of the discretised optimal state $y_h$, the adjoint state $p_h$
and the optimal control $u_h$ in Figure \ref{fig:2} on two sequences of meshes
after $7$ and $6$ regular refinements with $\beta =2$, respectively. It is also observed that second order
convergence rate holds for $y_h$, $p_h$ and $u_h$. We remark that the solutions by the
multilevel correction method are almost the same as the results by the direct
optimization problem solving on the same meshes.

The algorithm is obviously more efficient if $\beta$ is as large as possible, i.e., the coarse finite element 
space is as coarse as possible. To support this we also consider two sequences of meshes after $4$ regular refinements with $\beta =4$ based on the above mentioned two initial meshes. We show the errors of
 the discretised optimal state $y_h$, the adjoint state $p_h$ and the optimal control $u_h$ in 
 Figure \ref{fig:4}, second order convergence rates for the optimal control, the state and adjoint state can be observed.

\vskip0.8cm
\begin{figure}[ht]
\centering
\includegraphics[totalheight=5cm, angle =-90]{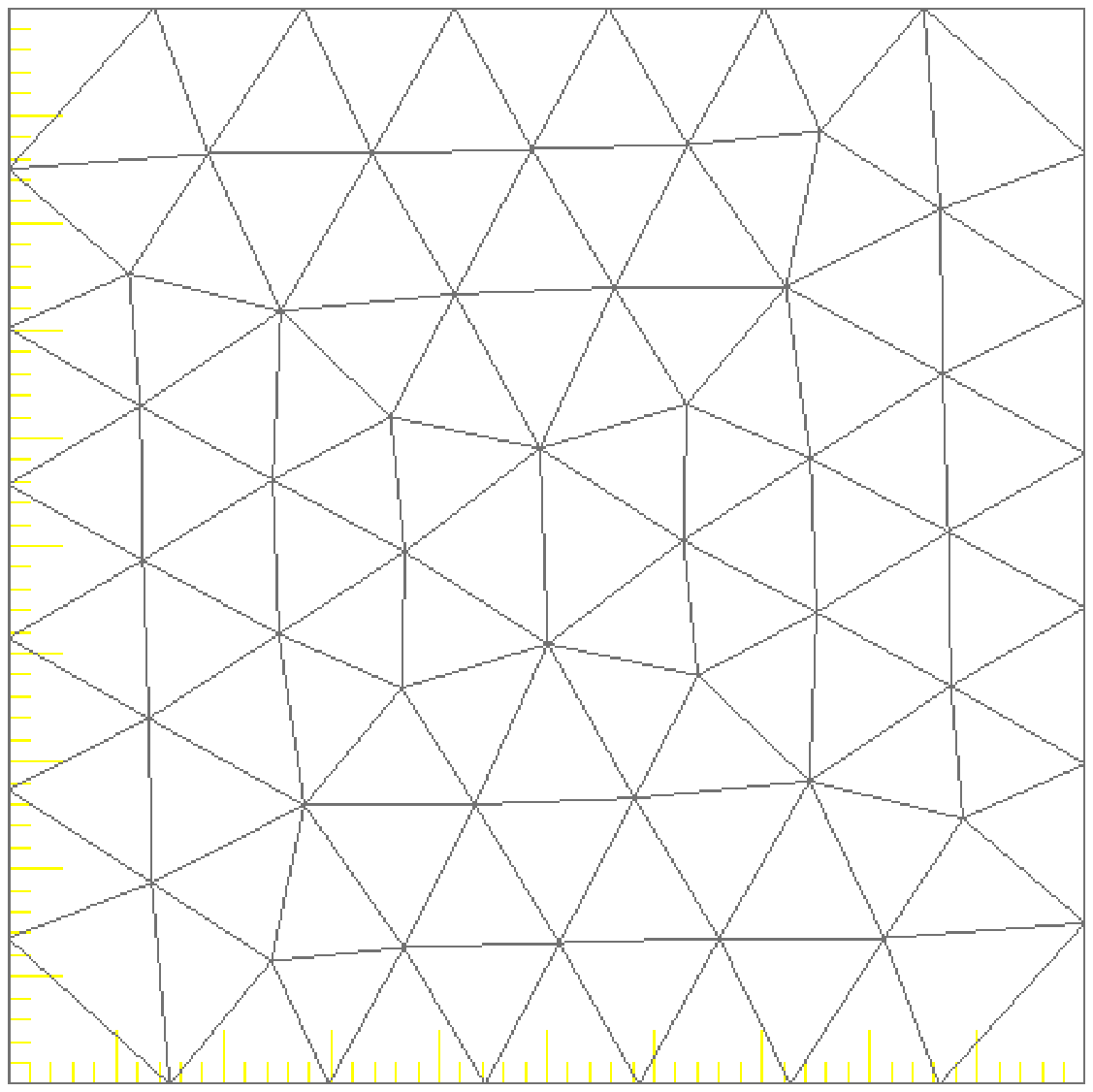}\hspace{3cm}
\includegraphics[totalheight=5cm, angle =-90]{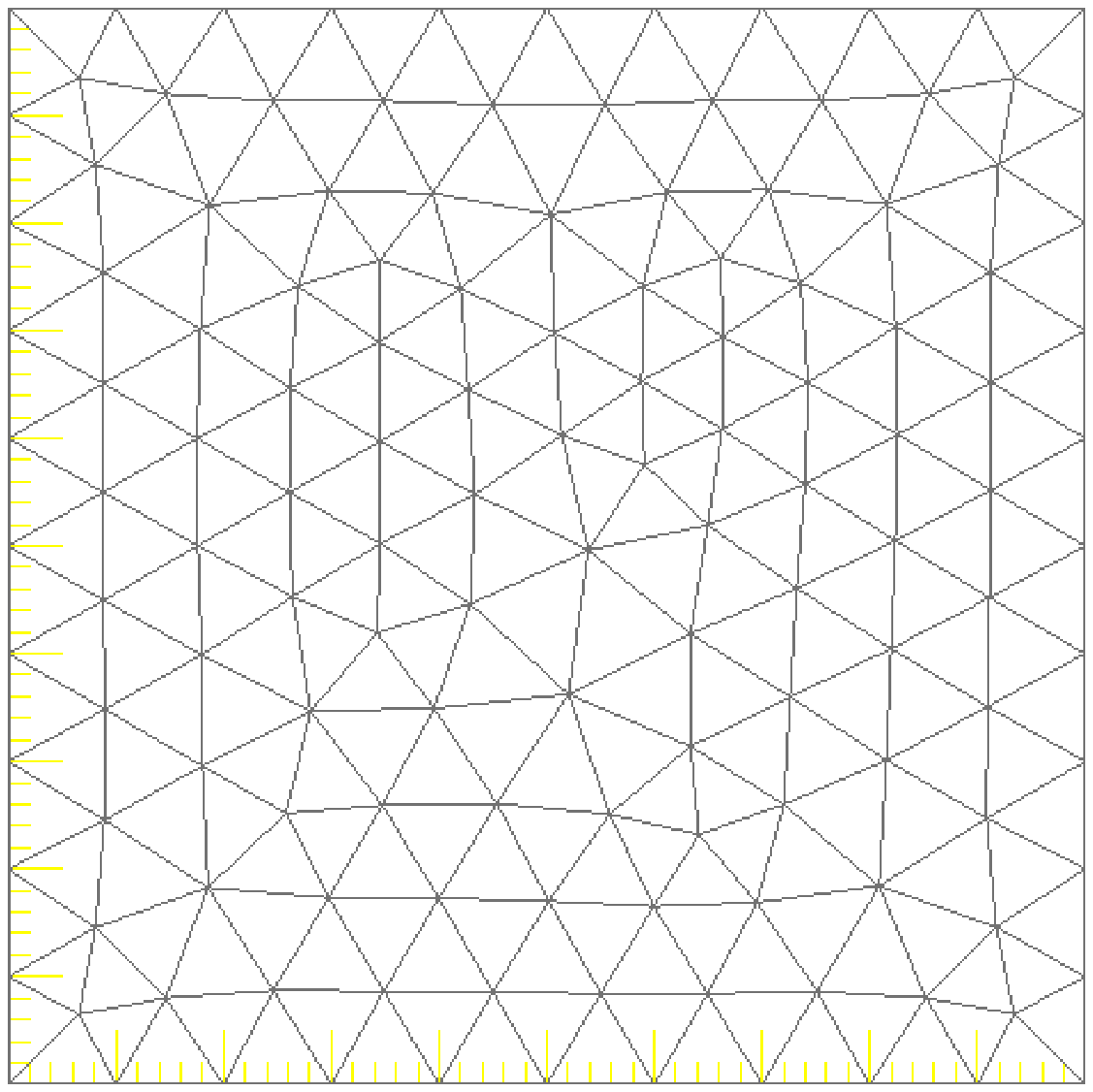}
\caption{Two initial meshes for the unit square with $68$
nodes (left) and $139$ nodes (right).}\label{fig:1}
\end{figure}

\begin{figure}[ht]
\centering
\includegraphics[width=7cm,height=7cm]{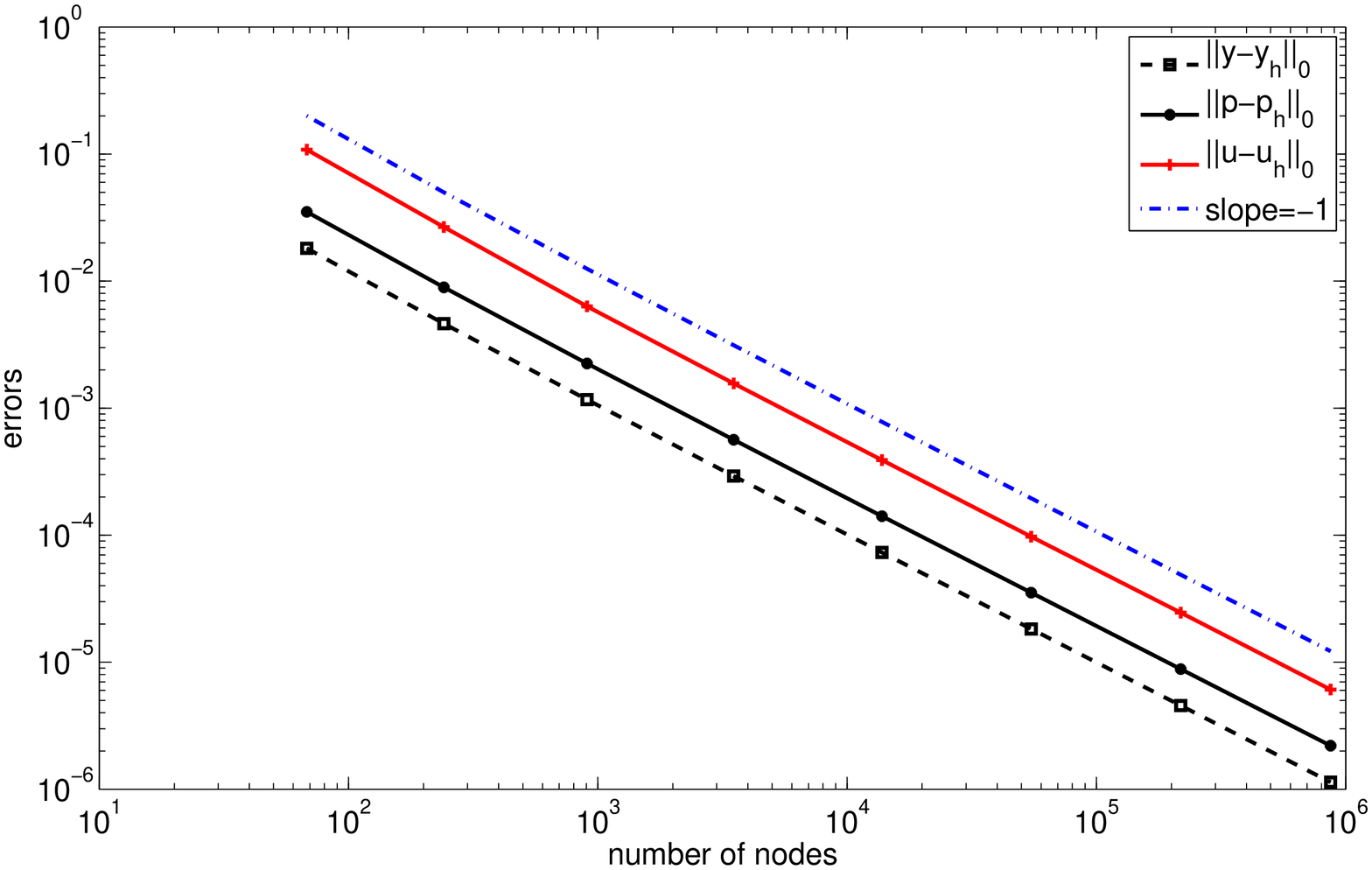}\hspace{0.3cm}
\includegraphics[width=7cm,height=7cm]{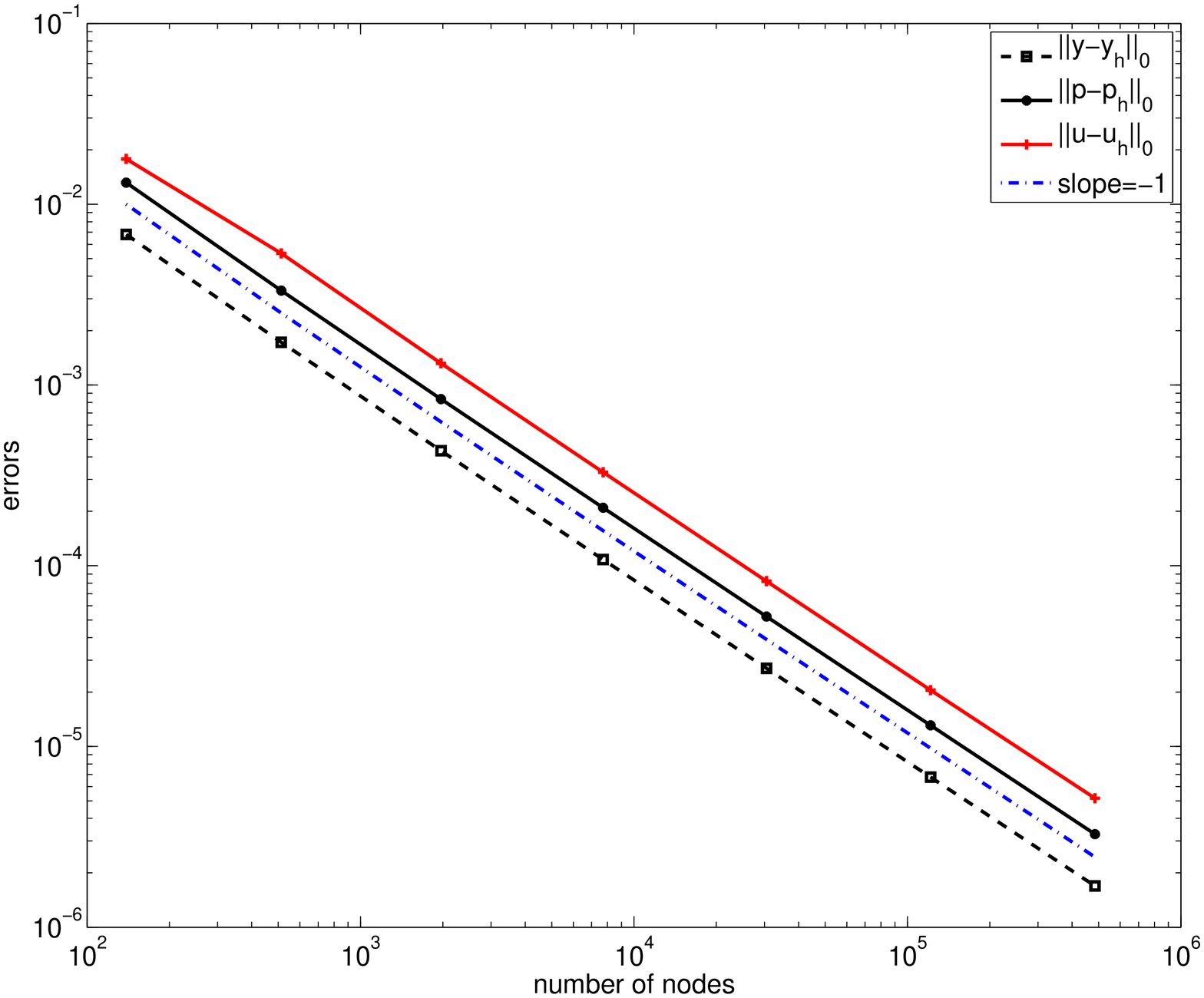}
\caption{Errors of the multilevel correction algorithm
for the discretised optimal state $y_h$, adjoint state $p_h$
and optimal control $u_h$  of Example \ref{Exm:1.1} (the left figure corresponds to the
left mesh in Figure \ref{fig:1} with 68 nodes as the initial mesh
and $7$ uniform refinements with $\beta=2$,
the right figure corresponds to the right mesh in Figure \ref{fig:1}
with $139$ nodes as the initial mesh and 6 uniform refinements with $\beta=2$).}
\label{fig:2}
\end{figure}

\begin{figure}[ht]
\centering
\includegraphics[width=7cm,height=7cm]{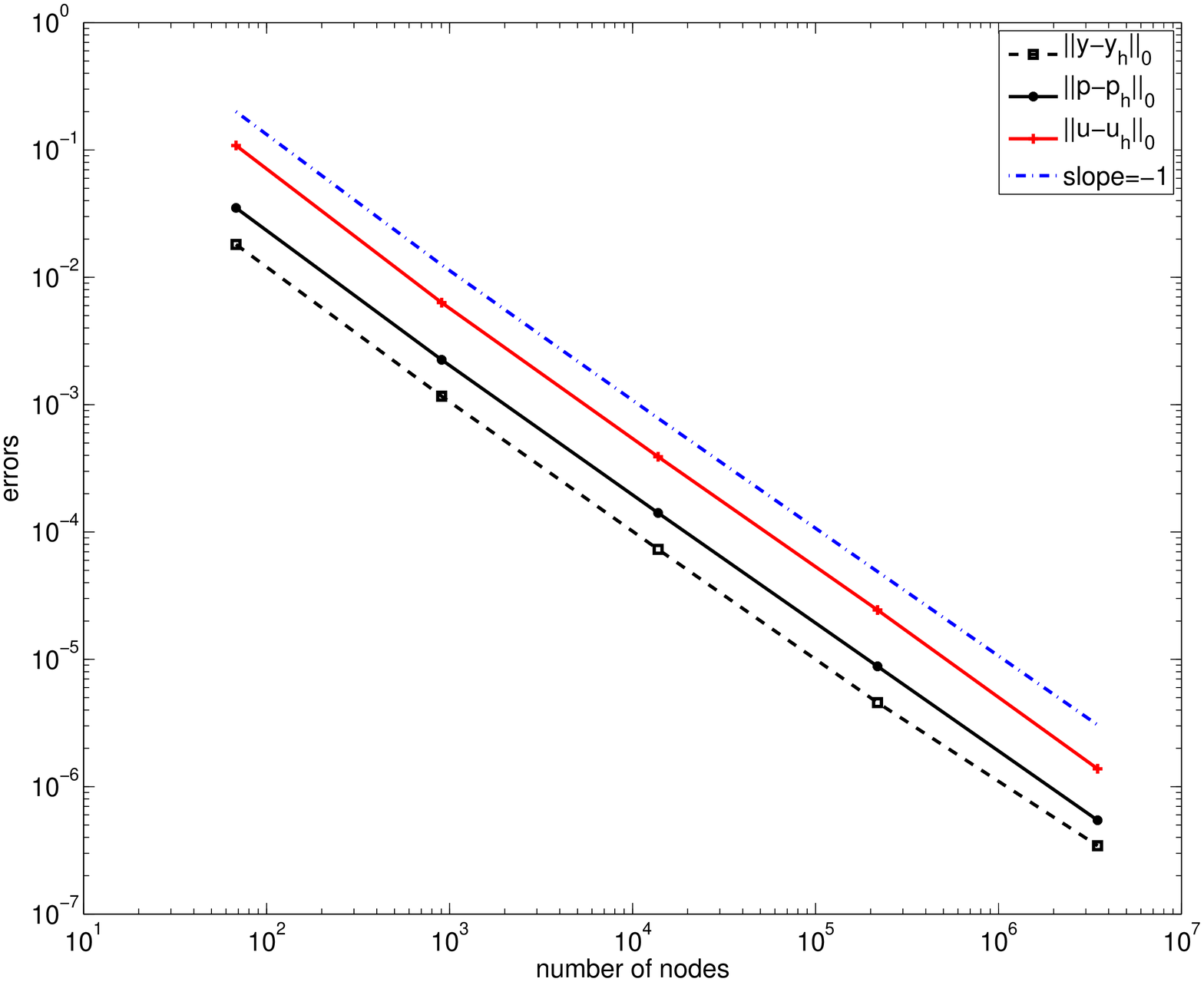}\hspace{0.3cm}
\includegraphics[width=7cm,height=7cm]{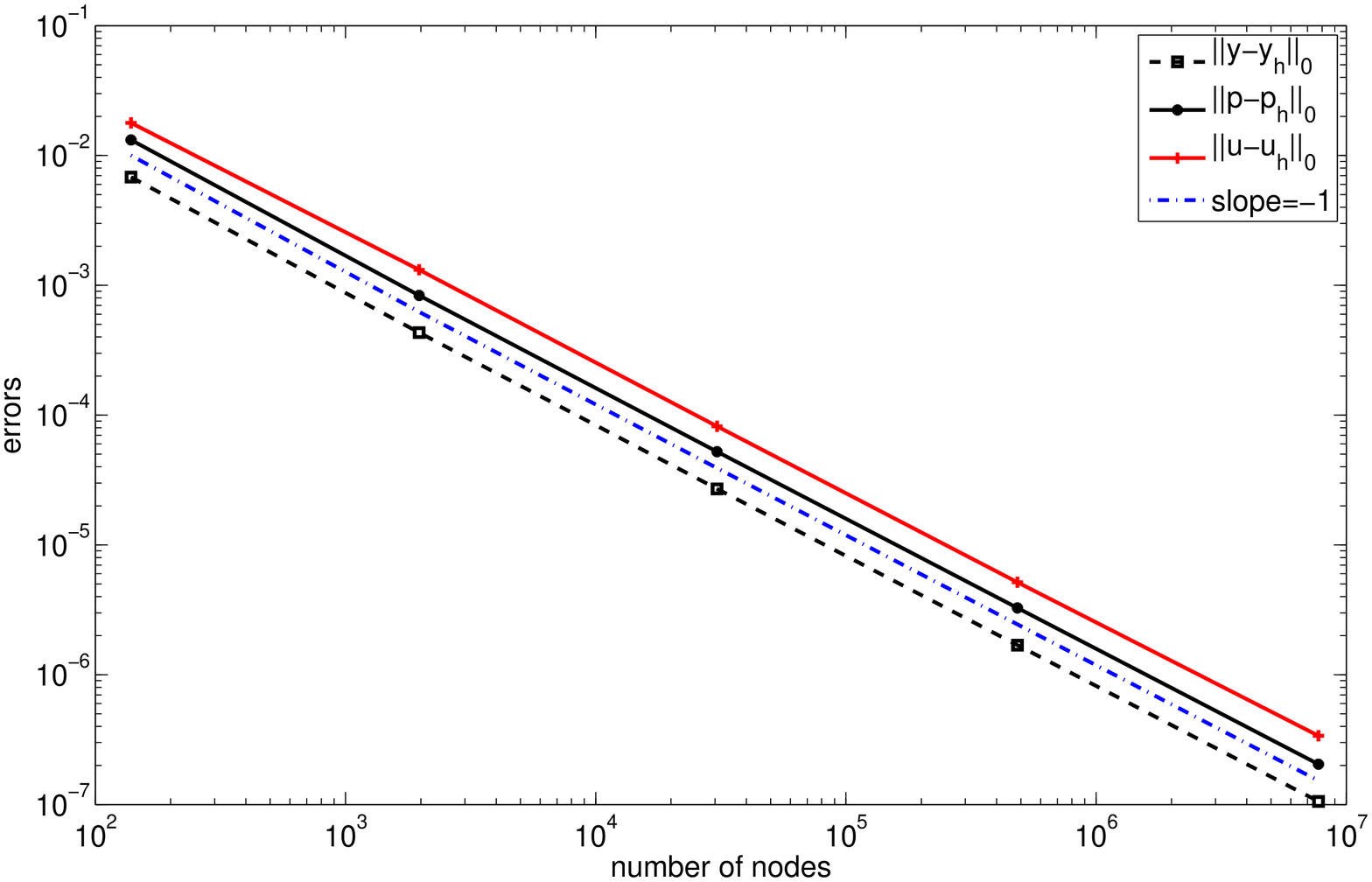}
\caption{Errors of the multilevel correction algorithm
for the discretised optimal state $y_h$, adjoint state $p_h$
and optimal control $u_h$  of Example \ref{Exm:1.1} (the left figure corresponds to the
left mesh in Figure \ref{fig:1} with 68 nodes as the initial mesh
and $4$ uniform refinements with $\beta=4$,
the right figure corresponds to the right mesh in Figure \ref{fig:1}
with $139$ nodes as the initial mesh and 4 uniform refinements with $\beta=4$).}
\label{fig:4}
\end{figure}

In the second example, we  consider the following optimal controls of semilinear elliptic equation:
\begin{eqnarray}
\min\limits_{u\in U_{ad}}\ \ J(y,u)={1\over
2}\|y-y_d\|_{0,\Omega}^2 +
\frac{\alpha}{2}\|u\|_{0,\Omega}^2\nonumber
\end{eqnarray}
subject to
\begin{equation}\label{example_state_1}
\left\{\begin{array}{llr} -\Delta y+y^3=f+u \ \ &\mbox{in}\
\Omega, \\
 \ y=0  \ \ \ &\mbox{on}\ \partial\Omega,\\
 a(x)\leq u(x)\leq b(x)\ \ &\mbox{a.e.\ in}\ \Omega.
\end{array} \right.
\end{equation}
\begin{Example}\label{Exm:1.2}
We set $\Omega=[0,1]^2$. Let
$a=0$, $b=3$, $\alpha =0.01$, $g_1(x_1,x_2)=2\pi^2\sin(\pi x_1)\sin(\pi x_2)$, 
$g_2(x_1,x_2)=\sin^3(\pi x_1)\sin^3(\pi x_2)$. Then $f$
is chosen as
\begin{eqnarray}\label{nonumber}
f(x_1,x_2) =\left\{\begin{array}{llr}
g_1(x_1,x_2)+g_2(x_1,x_2)-a,\ \ \ \ \ &g_1(x_1,x_2)< a;\\
g_2(x_1,x_2), \  \ \ \ \ &a\leq g_1(x_1,x_2)\leq b;\\
g_1(x_1,x_2)+g_2(x_1,x_2)-b,\ \ \ \ \ &g_1(x_1,x_2)>b.
\end{array}\right.
\end{eqnarray}
Due to the state equation (\ref{example_state_1}),
we obtain the exact optimal control $u$
\begin{eqnarray}\label{nonumber}
u(x_1,x_2) =\left\{\begin{array}{llr}
a,\ \ \ \ \ &g_1(x_1,x_2)< a;\\
g_1(x_1,x_2), \  \ \ \ \ &a\leq g_1(x_1,x_2)\leq b;\\
b,\ \ \ \ \ &g_1(x_1,x_2)>b.
\end{array}\right.
\end{eqnarray}
We also have
\begin{eqnarray*}
y(x_1,x_2)=\sin(\pi x_1)\sin(\pi x_2),\nonumber\\
p(x_1,x_2)=-2\pi^2\alpha\sin(\pi x_1)\sin(\pi x_2).
\end{eqnarray*}
The desired state is given by
\begin{eqnarray*}
y_d(x_1,x_2)=y(x_1,x_2)-3y^2p+4\pi^4\alpha\sin(\pi x_1)\sin(\pi x_2).
\end{eqnarray*}
\end{Example}

We solve this nonlinear optimisation problem with standard SQP method (see \cite{Hinze09book}).
 We test our proposed multilevel correction algorithm for the above nonlinear optimal control problems 
 on the sequence of multilevel meshes. We use the same sequence of meshes  generated with $\beta=2$ 
 as in the first example. It is also observed that second order
convergence rate holds for $y_h$, $p_h$ and $u_h$ in the nonlinear case. We remark that the solutions
 by the multilevel correction method are almost the same as the results by the direct
optimization problem solving on the same meshes.

\begin{figure}[ht]
\centering
\includegraphics[width=7cm,height=7cm]{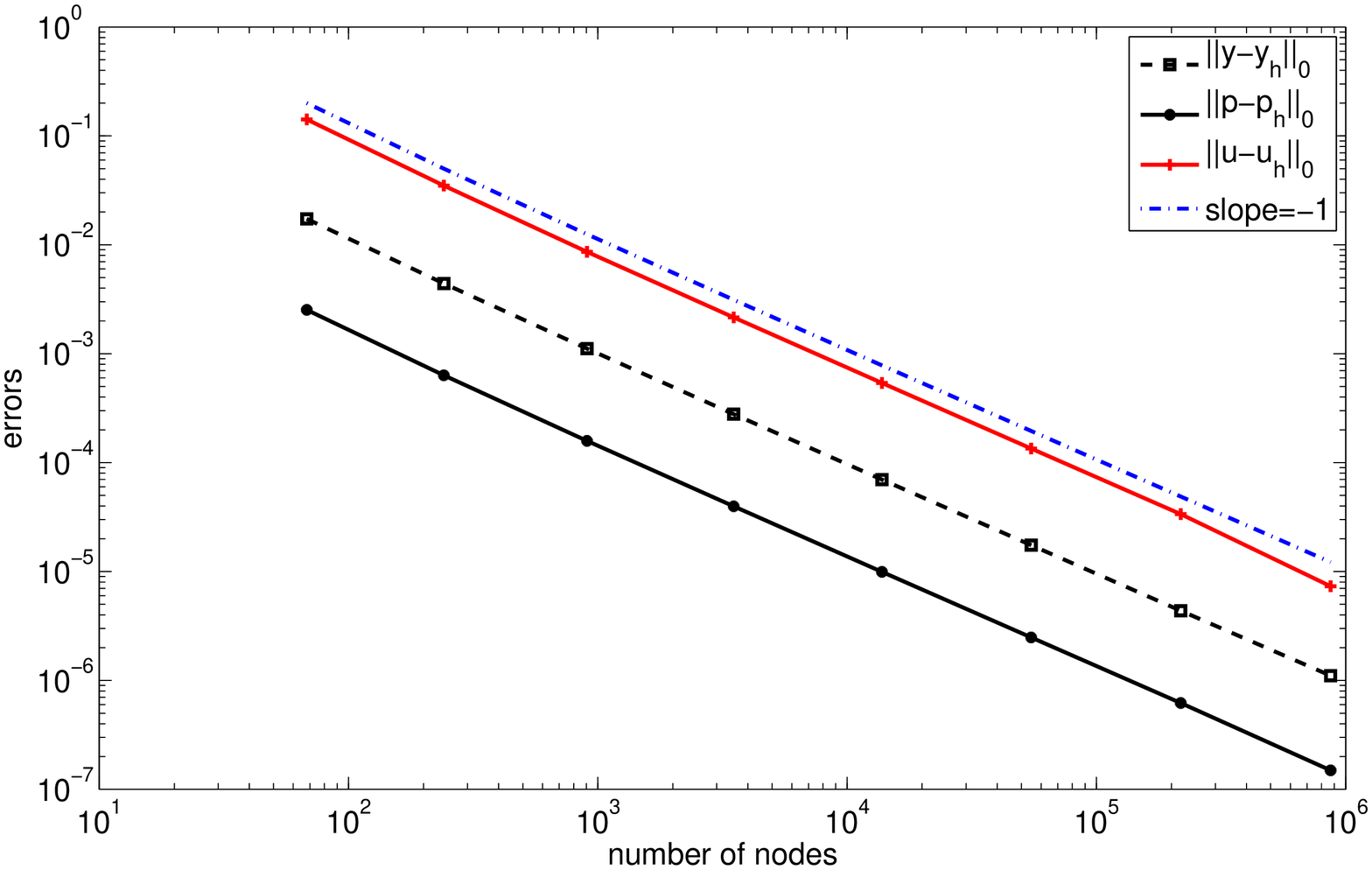}\hspace{0.3cm}
\includegraphics[width=7cm,height=7cm]{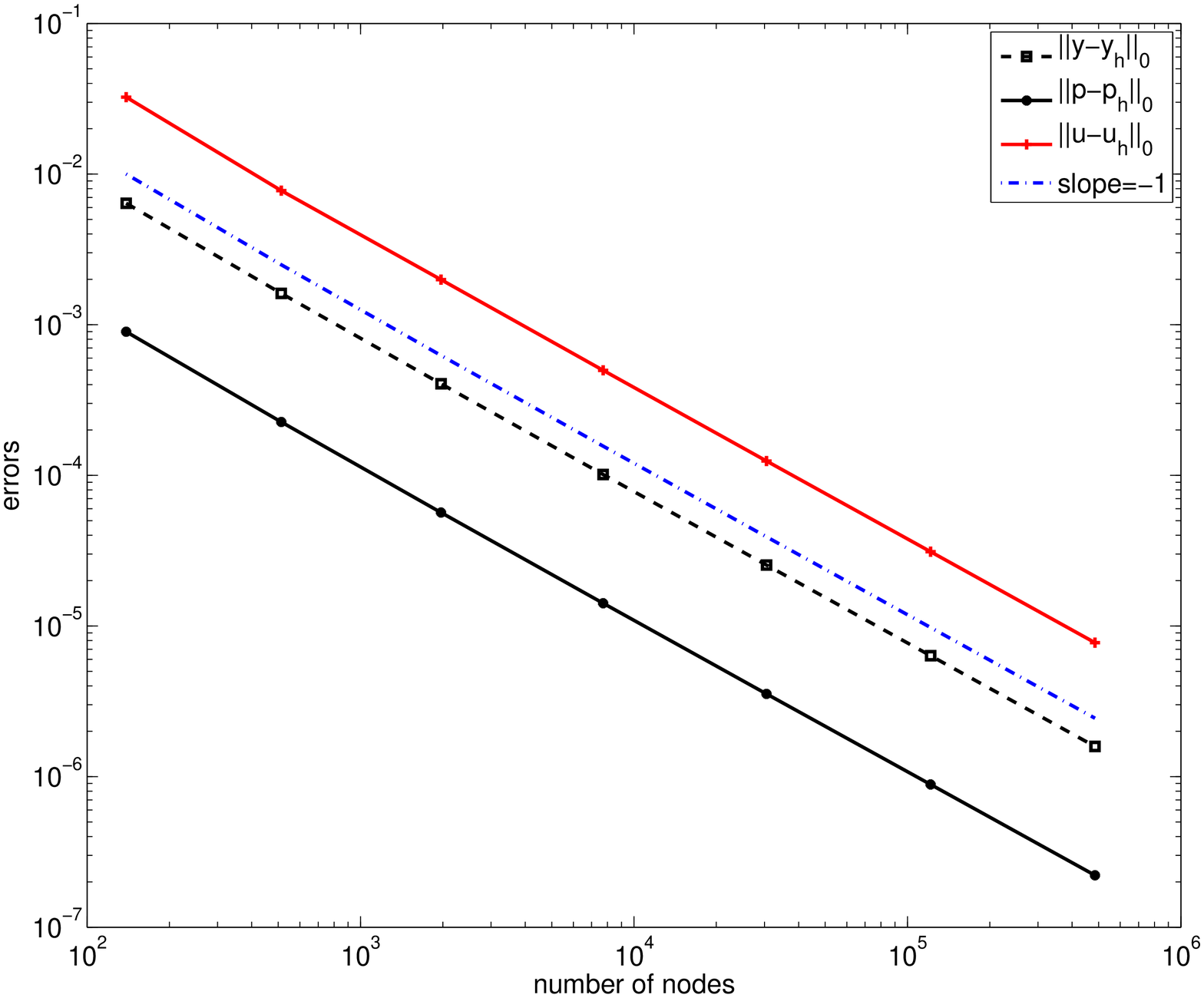}
\caption{Errors of the multilevel correction algorithm
for the discretised optimal state $y_h$, adjoint state $p_h$
and optimal control $u_h$ of Example \ref{Exm:1.2} (the left figure corresponds to the
left mesh in Figure \ref{fig:1} with 68 nodes as the initial mesh
and $7$ uniform refinements,
the right figure corresponds to the right mesh in Figure \ref{fig:1}
with $139$ nodes as the initial mesh and 6 uniform refinements).}
\label{fig:3}
\end{figure}

\section{Concluding remarks}
In this paper, we introduce a type of multilevel correction scheme
to solve the optimal control problem. The idea here is to use the
multilevel correction method to transform the solution of the optimal control
 problem on the finest finite element space to a series of solutions of the corresponding linear
boundary value problems which can be solved by the multigrid method and a
series of solutions of optimal control problems on the coarsest
finite element space. The optimal control problem solving is more difficult than
the linear boundary value problem solving which has already many efficient solvers. Thus,
the proposed method can improve the overall efficiency for the optimization problem solving.
With the complexity analysis, we can find that the multilevel correction
scheme can obtain the optimal finite element approximation by the almost optimal
computational work \cite{Xie_IMA,Xie_JCP}.

We can replace the multigrid method by other types of efficient solvers
such as algebraic multigrid method and the domain decomposition method.
Furthermore, the framework here can also be coupled with the
parallel method and the adaptive refinement technique. The ideas can be extended to
other types of linear and nonlinear optimal control problems.

\section*{Acknowledgments}
The first author was supported by the National Basic Research Program of China under grant 2012CB821204,
 the National Natural Science Foundation of China under grant 11201464 and 91330115, and the scientific 
 Research Foundation for the Returned Overseas Chinese Scholars, State Education Ministry. The second 
 author gratefully acknowledges the support of the National Natural Science Foundation of China 
 (91330202, 11371026, 11001259, 11031006, 2011CB309703), the National Center for Mathematics 
 and Interdisciplinary Science, CAS and the President Foundation of AMSS-CAS. The third author 
 was supported by the National Natural Science Foundation of China under grant 11171337.


\end{document}